\documentclass[10pt]{article}
\usepackage{amsfonts,amssymb,amsmath, amsthm}
\theoremstyle{plain}

\newtheorem*{assumptionA}{Assumption A}
\newtheorem{theorem}{Theorem}[section]
\newtheorem{definition}[theorem]{Definition}
\newtheorem{corollary}[theorem]{Corollary}
\newtheorem{lemma}[theorem]{Lemma}
\newtheorem{proposition}[theorem]{Proposition}
\newtheorem{example}[theorem]{Example}

\newtheorem{preremark}[theorem]{Remark}
\newenvironment{remark}{\begin{preremark}\normalfont}{\end{preremark}}
\allowdisplaybreaks

\renewcommand{\l}{\mathcal L}
\renewcommand{\r}{\mathcal R}
\newcommand{\e}{\mathcal E}
\newcommand{\F}{\mathcal F}

\newcommand{\R}{\mathbb R}

\title{The Dirichlet heat kernel in inner uniform domains: local results, 
compact domains and non-symmetric forms}
\author{Janna Lierl \and Laurent Saloff-Coste \thanks{Both authors partially supported by NSF Grant DMS 1004771}}

\begin{document}
\maketitle

\begin{abstract}
This paper provides sharp Dirichlet heat kernel estimates in inner uniform 
domains, including bounded inner uniform domains, 
in the context of certain (possibly non-symmetric) bilinear forms 
resembling Dirichlet forms.
For instance, the results apply to the Dirichlet heat kernel associated with 
a uniformly elliptic divergence form operator with symmetric second order part
and bounded measurable real coefficients 
in inner uniform domains in 
$\mathbb  R^n$. The results are applicable to any convex domain, to the complement of any convex domain, and 
to more exotic examples such as the interior and exterior of the snowflake.  
\end{abstract}



\section{Introduction}

This paper is concerned with Dirichlet heat kernel estimates for diffusions
in inner uniform domains. 
The monograph \cite{GyryaSC} introduced a general approach 
to this problem in the case of unbounded domains in  strongly local 
Dirichlet spaces 
satisfying a global parabolic Harnack inequality. 
Sharp estimates for 
the heat kernel and the heat semigroup with Dirichlet boundary condition in 
domains have been studied by many authors. The article \cite{DavSim} 
contains seminal ideas.  Varopoulos' work \cite{VarLip1,VarLip2} contains 
definitive results for domains above the graph of a Lipschitz function. We refer 
the reader to \cite{GSC02,KimSong,Ouhbook,OuhW,Song1} for related results and 
further pointers to the literature. The main difference between these earlier 
works and the present effort is twofold. First, as in \cite{GyryaSC}, 
our results cover inner uniform domains, a class of domains that is 
significantly larger than, say, Lipschitz domains. Further, inner uniformity
is an intrinsic notion that can be used in rather general metric spaces. 
This allows us to develop our results in the context of a large class of 
local Dirichlet spaces. This larger context allows us to cover many natural 
and interesting examples beyond elliptic operators in $\mathbb R^n$, 
for instance, sub-elliptic operators.

This paper 
complements the results of \cite{GyryaSC} in several significant ways.
For this purpose, we rely heavily  on
key results contained in the companion papers \cite{LierlSC1,LierlSC2}
that were developed with the applications given here in mind. 

First, we treat 
the case of bounded inner uniform domains which is not covered by 
\cite{GyryaSC}. 
In the unbounded case, a
Doob's transform is used which involves the ``harmonic profile'' $h_U$ of 
the domain $U$, 
that is, a harmonic positive function in $U$ that vanishes on the boundary 
(in the proper sense).  In the case of bounded inner uniform domains, $h_U$ 
must be replaced by the positive eigenfunction $\phi_U$ associated with 
the lowest Dirichlet eigenvalue $\lambda_U$ of the domain $U$. This requires 
significant adaptation of the arguments. 

Second, whether the domain is bounded or not, we include 
a wide class of non-symmetric second order differential operators.
In the case of a fixed bounded inner uniform domain, there is not much  
difference in the final results between the symmetric and non symmetric cases. 
In the case of unbounded domains, the presence of lower order terms forces the
estimates to be local in time (in a certain sense).

Third, in both  the symmetric and  non-symmetric cases, we relax 
the very global assumptions made in \cite{GyryaSC} to cover cases where the 
geometry of the underlying space is only  controlled locally. In particular, 
we cover domains that are inner uniform only in a certain local sense. 
For instance, we treat the Dirichlet heat kernel for the 
Laplace-Beltrami operator in an unbounded inner uniform domain in a 
complete Riemannian manifold, without global curvature assumption, or
under the Ricci curvature assumption  $\mbox{Ric} \ge -\kappa g$, for some 
$\kappa>0$.  We also obtain some local estimates for the 
Dirichlet heat kernel in the interior of an unbounded convex set 
in $\mathbb R^n$.  Most unbounded convex sets are not inner uniform 
but they are always locally inner uniform (in fact, locally uniform).
 
We will work in a rather abstract setting involving the notion of
(not necessarily symmetric) Dirichlet forms and the associated intrinsic 
distance. This setting is actually very natural for this problem because, 
even when treating domains in $\mathbb R^n$, the technique we use requires 
the introduction of some auxiliary abstract Dirichlet spaces in which most of 
the work is done.  Regarding the general theory of Dirichlet spaces, we refer the reader to \cite{CF12,FOT94} and also \cite{MR92,MaRock}.
Nevertheless, in the rest of this introduction, 
we illustrate the main results of this paper in  the context of certain
elliptic operators on a complete Riemannian manifold.  

\subsection{Illustrative examples}
Let $(M,g)$ be a complete Riemannian manifold equipped with its Riemannian 
measure $\mu$ and its Riemannian distance function. Let $U$ be an inner uniform 
domain in $M$  (for instance, if $M=\mathbb R^n$, 
bounded convex domains are inner uniform 
and the complement of any convex domain is inner uniform). Let $L$ be a second 
order differential operator on $M$ of the form
$$ L= \Delta  +X  +V$$  
where $\Delta$ is the Laplace-Beltrami operator on $M$, 
$X$ is a smooth vector field on $M$ 
(viewed as a differential operator acting on smooth functions 
$X: f\mapsto Xf=df(X)$) 
and $V$ is a smooth function on $M$ (viewed as a multiplication operator). 
This particular structure of the differential operator $L$ is chosen here 
for convenience and illustrative purpose.  Given a domain $U$ in $M$, 
let $d_U$ be the inner distance in $U$ (see Section \ref{sec-innermetric} below).

Suppose that $M$ has non-negative Ricci curvature and $X=0, V=0$. Suppose 
also that $U$ is unbounded. Then \cite{GyryaSC} provides a global space-time
two-sided estimate of the Dirichlet heat kernel $h^D_U(t,x,y)$ of the form
$$C\frac{h_U(x)h_U(y)}{
\sqrt{V(x,\sqrt{t})V(y,\sqrt{t})} h_U(x_{\sqrt{t}})h_U(y_{\sqrt{t}})}
\exp\left(-c \frac{d_U(x,y)^2}{t}\right).$$ 
In this two-sided estimate, different constants $C,c\in (0,\infty)$ 
are used in in the lower and upper bounds. The function $h_U$ is any fixed 
positive solution of $Lh=0$ in $U$ which 
vanishes at the boundary (in the proper weak sense). 
We call this function a harmonic profile for $U$. For any $x\in U$ 
and $r>0$, $x_r$ denotes a point in $U$ with the property that 
$d(x,x_r) \le  Ar $ and $d(\partial U, x_r)\ge a r$ where $a,A$ are independent 
of $x$ and $r$. The inner uniformity of $U$ 
ensures that there exists constants $a,A$  such that such 
a point $x_r$ exists for every $x\in U$ and $r>0$.

The aim of this paper is to prove the theorems of the following type.
See Theorem \ref{th-bounded1} and Corollary \ref{cor-bounded1}.
\begin{theorem} Let $(M,g)$ be a complete Riemannian manifold 
with Riemannian measure $\mu$. Let $L=\Delta+X+V$ be as described above.
Let $U$ be a bounded inner uniform domain in $M$. 
Let $A=A(U),a=a(U)$ 
be constants such that for any point $x$ in $U$ and any $r>0$,
there exists a point $x_r$ in $U$ at distance at most $A\min\{r,1\}$ 
from $x$ and at distance at least $a\min\{r,1\}$ from the boundary of $U$.
Let $\phi_{\mbox{\tiny s}}$ 
(resp. $\phi$) be the unique positive eigenfunction associated with the lowest 
Dirichlet eigenvalue of $-\Delta$ (resp. $-L$) in $U$. 
\begin{itemize}
\item  There are constants $C=C(L,U),c=c(L,U)$ such that
$ c\phi_{\mbox{\tiny s}}
\le   \phi  \le C \phi_{\mbox{\tiny s}}$  in $U$.
\item  There are constants $C=C(L,U)$ and $\alpha=\alpha(L,U)$ such that,
for any solution $\psi$ of $L\psi= \lambda_\psi \psi$ in $U$ with Dirichlet
boundary condition,  we have  $|\psi|\le C(1+|\lambda_\psi|)^\alpha \phi$.  
\item For any fixed $T>0$, there are constant $c_i=c_i(L,U,T)\in (0,\infty)$ 
such that the Dirichlet heat kernel $p^{D}_{U}(t,x,y)$ for $L$ in $U$ 
with respect to $\mu$ satisfies
$$p^{D}_{U}(t,x,y)\le
\frac{ c_1 \phi(x)\phi(y)}{ 
\sqrt{V(x,\sqrt{t})V(y,\sqrt{t})} \phi(x_{\sqrt{t}})\phi(y_{\sqrt{t}})}
\exp\left(-c_2 \frac{d_U(x,y)^2}{t}\right)$$ 
and
$$p^{D}_{U}(t,x,y)\ge
\frac{ c_3 \phi(x)\phi(y)}{ 
\sqrt{V(x,\sqrt{t})V(y,\sqrt{t})} \phi(x_{\sqrt{t}})\phi(y_{\sqrt{t}})}
\exp\left(-c_4 \frac{d_U(x,y)^2}{t}\right),$$ 
for all  $(t,x,y)\in (0,T)\times U\times U$.
\end{itemize}    
\end{theorem}
To our knowledge, this theorem is new even when $M=\mathbb R^n$ and 
$L=\Delta$ is the Laplacian. Indeed, \cite{GyryaSC} does not treat bounded 
domains and, even in this special case, the above statement is more precise 
than the known intrinsic ultracontractivity results. 
Section \ref{sec-conv} gives more detailed results in a 
more general context and include complementary asymptotics when 
$t$ tends to infinity.  In particular, Corollary \ref{cor-bounded1} gives a
refined eigenfunction estimate.

\begin{figure}[h]
\caption{A polygonal domain $\Omega$ with a slit}
\begin{center}\label{fig1}

\begin{picture}(300,90)(-50,30)\thicklines
\put(100,50){\line(2,-1){50}}
\put(150,25){\line(0,1){45}}
\put(80,70){\line(0,1){50}}
\put(80,120){\line(1,0){30}}
\put(80,70){\line(-3,1){45}}
\put(100,70){\line(1,0){50}}
\put(100,70){\line(1,5){10}}
\put(100,50){\line(1,-1){40}}
\put(140,10){\line(-1,0){80}}
\put(60,10){\line(-1,3){25}}
\put(50,40){\line(1,0){30}}
\put(90,60){\makebox(0,0){$\Omega$}}
\end{picture}
\end{center}
\end{figure}
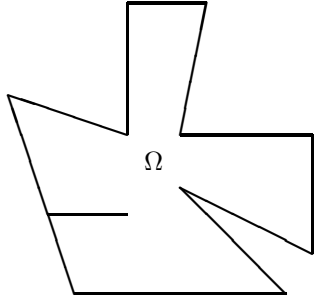

\vspace{.1in}

For very concrete examples, the reader can think of a bounded 
polygonal domain $\Omega$ in 
$\mathbb R^n$ as in Figure \ref{fig1}. In this context, we can consider 
the heat equation with Dirichlet boundary condition for the divergence 
form operator
$$Lf=\sum \partial_i^2 f +\sum b_i\partial _i f+ \sum \partial_i(d_i f)+cf$$
where $b_i,d_i,c$ are bounded measurable functions. Let $\phi$ be the positive 
eigenfunction associated with the lowest Dirichlet 
eigenvalue of $-L$ in $\Omega$. Let $\phi_{\mbox{\tiny s}}$ be the positive
eigenfunction associated with the lowest Dirichlet 
eigenvalue of $-\sum \partial_i^2$ in $\Omega$. We show that 
$\phi \simeq \phi_{\mbox{\tiny s}}$
in $\Omega$. The function $\phi_{\mbox{\tiny s}}$
vanishes at different rates 
as $x$ tends non-tangentially to different boundary points. The rate 
depends on the angle at the boundary point. For instance,
$\phi$ will vanish linearly  at smooth boundary points and will vanish 
quadratically  when approaching the vertex of an 
interior right angle. The polygonal domain $\Omega$ may have a vertex 
with interior angle of  $2\pi$ in which case the corresponding vertex is 
the tip of a slit. At such a vertex, $\phi$ vanishes as the square root of 
the distance to the boundary.  The heat kernel estimates stated above 
capture this in a very precise way by reducing the estimates of the 
Dirichlet heat kernel to the understanding of the eigenfunction $\phi$ 
(equivalently, $\phi_{\mbox{\tiny s}}$). 
The case of the Koch snowflake is another  good example to keep in mind.

An important special case of the results obtained in this paper
arises when the manifold $M$ 
has non-negative 
Ricci curvature (hence satisfies the parabolic Harnack inequality at all scales)
and $L=\Delta$. In this  case, the results described above hold true 
uniformly over the class of all inner uniform domains with  specified
inner uniformity constants a stated in the following theorem.

\begin{theorem} Let $(M,g)$ be a complete non-compact Riemannian manifold with 
non-negative Ricci curvature. Fix constants $0<c_u<1<C_u<\infty$ and 
let $U$ be a bounded $(c_u,C_u)$-inner uniform domain in $M$ (see Definition
\ref{def-IU}). 
Let $\mbox{\em diam}_U$ be the inner diameter of $U$. Let $\lambda_U$
the the lowest eigenvalue of minus the Laplacian with Dirichlet boundary 
condition in $U$, and let  $\phi$ be the associated positive eigenfunction normalized in $L^2(U,\mu)$.
For any $x\in U$, let $x_r\in U$ be such that $d_U(x,x_r)\le r$ 
and $d(x_r,\partial_U)\ge 2^{-5}c_u\min\{r,
\mbox{\em diam}_U\}$. 
Let $p^{D}_{U}(t,x,y)$ be the Dirichlet heat kernel in $U$.
There are constants $c_i\in (0,\infty)$ depending only on $M$ and $c_u,C_u$
such that 
\begin{itemize}
\item The Dirichlet heat kernel satisfies
$$p^{D}_{U}(t,x,y)\le
\frac{ c_1e^{-t\lambda_U}\phi(x)\phi(y)}{ 
\sqrt{V(x,\sqrt{\tau})V(y,\sqrt{\tau})} \phi(x_{\sqrt{\tau}})
\phi(y_{\sqrt{\tau}})}
\exp\left(-c_2 \frac{d_U(x,y)^2}{t}\right)$$ 
and
$$p^{D}_{U}(t,x,y)\ge
\frac{ c_3e^{-t\lambda_U}\phi(x)\phi(y)}{ 
\sqrt{V(x,\sqrt{\tau})V(y,\sqrt{\tau})} \phi(x_{\sqrt{\tau}})
\phi(y_{\sqrt{\tau}})}
\exp\left(-c_4 \frac{d_U(x,y)^2}{t}\right),$$ 
for all  $(t,x,y)\in (0,\infty)\times U\times U$ with $\tau=
\min\{t,\mbox{\em diam}_U^2\}$.
\item Further,  for $(t,x,y)\in (\mbox{\em diam}_U^2,\infty)\times U\times U$,
we have
$$\left|\frac{e^{t\lambda_U}p^D_U(t,x,y)}{\phi(x)\phi(y)}-1\right|\le c_5e^{-c_6t/\mbox{\small \em diam}_U^2}.$$ 
\end{itemize}
\end{theorem}

As a simple example of application of this result, let $M=\mathbb R^n$ 
be the Euclidean space. Let $\mathcal C(a,A)$ be the set of all 
convex bounded regions $U$ such $B(o,ar)\subset U\subset B(o,Ar)$ 
for some $o\in U$ and $r>0$. It is not hard to see that 
there are constants $c_u,C_u$, depending only on $a,A$, 
such that any such set is $(c_u,C_u)$-inner uniform. The above theorem applies 
uniformly to all $U\in \mathcal C(a,A)$.

The general setting in which we will work allows us to cover 
many different situations including the case when the Riemannian structure 
used above
is replaced by a sub-Riemannian structure.

\if

Consider the case of a uniformly elliptic operator in $\mathbb R^n$, i.e., 
an operator
formally given by 
$Lf=\sum_{i,j}\partial_i (a_{i,j} \partial_jf)+\sum_i (b_i\partial _i f+
\partial_i(d_if))+cf$
with bounded 
measurable real  coefficients $a_{i,j}$, $b_i,d_i,c$ with 
$\sum_{i,j}a_{i,j}(x)\xi_i\xi_j\ge \epsilon |\xi|^2$, $\epsilon>0$. 
The results of this 
paper provide good bounds for the heat kernel $p^D_U(t,x,y)$ associated with the 
parabolic equation $\partial_t f=Lf$ with Dirichlet boundary 
condition in any inner uniform domain $U$, under the hypothesis that 
the matrix $(a_{i,j})$ is symmetric at any point $x$. 
The first author will consider the case of a non-symmetric $(a_{i,j})$ 
elsewhere.
\fi 

\subsection{Organization of the paper}
In the next section, we describe  basic notation and assumptions regarding 
the underlying space $X$ and its geometry induced by a fixed strongly 
local Dirichlet.  The doubling volume property and Poincar\'e inequalities 
play a key role throughout the paper.

Section 3 contains the definition of uniform and inner uniform domains as 
well as important local quantitative version.  

Section 4 described a class of bilinear forms with dense domain in $L^2(X,\mu)$
that are adapted to the 
fix geometric structure carried by our space $X$. See Definition \ref{def-adapt}
and Assumption A. For example, 
if $X$ is a complete Riemannian manifold Riemannian measure $\mu$ and 
Dirichlet form  $\int_M \nabla f_1 \cdot \nabla f_2 d\mu$ then the bilinear form
\begin{eqnarray*}
\mathcal E(f_1,f_2) &=& \int_M \nabla f_1 \cdot \nabla f_2 d\mu\\
&&+ \int_M  (b_1\cdot \nabla f_1) f_2 d\mu + \int_M  f_1 (b_2\cdot \nabla f_2) f_2 d\mu 
+\int_M f_1f_2 Vd\mu \end{eqnarray*}  
where  $b_1,b_2$ are  bounded vector fields on $X$  and 
$V$ is a bounded potential is adapted in the sense introduced in Section 4. 

Section 5 discusses the notions of interior and boundary Harnack inequalities
and the notion of harmonic profile of a region $U$. 
The harmonic profile of an unbounded domain $U$ is a positive harmonic 
function in $U$ satisfying the Dirichlet boundary condition along 
the boundary of $U$.  A localized version of this definition is also introduced
and the existence of harmonic profiles is discussed. Results from
\cite{LierlSC1} that play an important role here are reviewed.

Section 6 provides novel variations on the notion of $h$-transform. It contains 
some of the key ingredients for the proof of our main Dirichlet heat kernel 
estimates. The main point is to understand the structure and properties  
of the form $\mathcal E_h$ obtains via $h$-transform from 
our given adapted bilinear form $\mathcal E$. 
Even if we assume that 
$\mathcal E$ is a (non-symmetric) Dirichlet form, 
the form $\mathcal E_h$ may not be a Dirichlet form.
The precise properties of $\mathcal E_h$ depend on the  particular function 
$h$ used in the $h$-transform. We show that, for well chosen $h$, the 
form $\mathcal E_h$ satisfies structural properties that imply the validity of 
a Harnack inequality (up to the boundary).  See Theorem \ref{h-VDP0} and 
Theorem \ref{h-VDP}. 
This makes use of the results of 
\cite{LierlSC2} which were developed in part for this purpose and are the 
main key to obtain the result presented here.

Section 7 contains the main results obtained in this paper. It is based in 
an essential way on the ideas and techniques described in Section 5 and 6. 
Theorems \ref{th-HK1}--\ref{th-HK3} 
provide detailed Dirichlet heat kernel estimates
covering a wide range of different hypotheses.  Theorem \ref{th-globH}
gives a global Harnack type estimate  for weak solutions
of our abstract heat equations with Dirichlet boundary condition under 
a range  of inner uniformity conditions on the domain.

\section{The underlying space and its geometry}\label{sec-geo}

\subsection{The intrinsic distance}
Let $X$ be a connected, locally compact, separable metrizable space and let $\mu$ be a 
non-negative Borel measure on $X$ that is finite on compact sets and positive 
on non-empty open sets.

We fix a symmetric, strongly local, regular Dirichlet form 
$(\e^{\mbox{\tiny{s}}}
,\F=D(\e^{\mbox{\tiny{s}}}))$ on $L^2(X,\mu)$ with 
energy measure $d\Gamma$. We sometimes call this form ```the model form''.
By this we simply mean that this form serves to define the basic 
geometry of our space and the adapted forms introduced in 
Section \ref{sec3:non-sym forms}.

Recall that $d\Gamma$ is a measure-valued quadratic 
form defined by
 \[ \int f \, d\Gamma(u,u) = \e^{\mbox{\tiny{s}}}
(uf,u) - \frac{1}{2} \e^{\mbox{\tiny{s}}}(f,u^2), 
\quad \forall f,u \in \F \cap L^{\infty}(X,\mu), \]
and extended to unbounded functions by setting $\Gamma(u,u) = \lim_{n \to \infty} \Gamma(u_n,u_n)$, where $u_n = \max\{\min\{u,n\},-n\}$. 
Using polarization, we obtain a bilinear form $d\Gamma$. In particular,
\[ \e^{\mbox{\tiny{s}}}(u,v) = \int d\Gamma(u,v), \quad \forall u,v \in \F. \]
We equip the Hilbert space $\F$ with its natural norm 
$$\|f\|_\F=\left(\int_X |f|^2d\mu +\int d\Gamma(f,f)\right)^{1/2}.$$

Let $U \subset X$ be an open set. Define
 \[ \F_{\mbox{\tiny{{loc}}}}(U)  =  \{ f \in L^2_{\mbox{\tiny{loc}}}(U) : \forall \textrm{ compact } K \subset U, \ \exists f^{\sharp} \in \F,
                      f = f^{\sharp}\big|_K \textrm{ a.e.} \} \]
For $f,g \in \F_{\mbox{\tiny{loc}}}(U)$ we define $\Gamma(f,g)$ locally by 
$\Gamma(f,g)\big|_K = \Gamma(f^{\sharp},g^{\sharp})\big|_K$, where $K \subset U$ 
is open relatively compact and $f^{\sharp},g^{\sharp}$ are functions in $\F$ 
such that $f = f^{\sharp}$, $g = g^{\sharp}$ a.e.~on $K$. Set
\begin{align*}
 \F(U)   &=  \{ u \in \F_{\mbox{\tiny{loc}}}(U) : \int_U |u|^2 d\mu + \int_U d\Gamma(u,u) < \infty \}, \\
\F_{\mbox{\tiny{c}}}(U)  
&=  \{ u \in \F(U) : \textrm{ the essential support of } u \textrm{ is compact in } U \}.\\
\F^0(U)  &= \mbox{ the closure of } \F_{\mbox{\tiny{c}}}(U) \mbox{ for the norm }
\left(\int_U |u|^2 d\mu + \int_U d\Gamma(u,u)\right)^{1/2} . 
\end{align*}

\begin{definition}
The \emph{intrinsic distance} $d := d_{\e^{\mbox{\em\tiny s}}}$ induced by 
$(\e^{\mbox{\em\tiny s}},\F)$ 
is defined as
\[ d_{\e^{\mbox{\emph{\tiny{s}}}}}(x,y) := \sup \big\{ f(x)-f(y): f \in 
\F_{\mbox{\em\tiny{loc}}}(X) \cap C(X), \, d\Gamma(f,f) \leq d\mu \big\},  \]
for all $x,y \in X$, where $C(X)$ is the space of continuous functions on $X$.
\end{definition}

Throughout this paper, the spaces $\F,\F(U), \F_{\mbox{\tiny{c}}}(U), \F^0(U) $ 
and the intrinsic distance $d$ play an essential role. The space $\F$ is the 
equivalent of the Sobolev space of $L^2$ functions with gradient in $L^2$.
The distance $d$ defines the geometry of our space and will be used to introduce  fundamental assumptions.

Consider the following properties of the intrinsic distance that may or may not be satisfied. They are discussed in \cite{Stu95geometry, SturmI}.
\begin{itemize}
\item (A1)  The intrinsic distance  $d $ is finite everywhere 
and defines the original topology of $X$.
\item (A2) The space $(X,d)$ is a complete metric space. 
\item (A2') $\forall x \in X, r > 0$, the open ball $B(x,r)$ 
is relatively compact in  $(X,d)$.
\end{itemize}
Note that if (A1) holds true then, 
by \cite[Theorem 2]{Stu95geometry}, (A2) is  equivalent to (A2').
Moreover, (A1)-(A2) imply that $(X,d)$ is a geodesic space, i.e., 
any two points in $X$ can be connected by a minimal geodesic in $X$. 
See \cite[Theorem 1]{Stu95geometry}.
If (A1) and (A2) hold true then 
the intrinsic distance is also given by (see \cite[Proposition 1]{SturmI}) 
\[  d(x,y) = \sup \big\{ f(x)-f(y): f \in \F \cap C_{\mbox{\tiny{c}}}(X), \, d\Gamma(f,f) \leq d\mu \big\}, \quad x,y \in X. \]
When working in an open subset $Y$ of $X$, it is sometimes 
sufficient to assume only (A1) and 
\begin{itemize}
\item (A2-$Y$) For any ball  $B(x,2r) \subset Y$,  $B(x,r)$ 
is relatively compact. 
\end{itemize}
This is a  version of property (A2') that is localized in a set $Y$ of particular interest. We will not pursue this systematically here but we will make a technical use of this fact at a later stage in the paper. 
In what follows we always assume that either (A1)-(A2) holds or, 
when justified by the context, that (A1)-(A2-$Y$) holds.

\begin{example} Let $\Omega$ be a domain in Euclidean space. 
Consider the (symmetric) Dirichlet form  
$\e_\Omega(f,f)=\int_\Omega |\nabla f|^2d\mu$ with domain $H^1_0(\Omega)$, the 
Sobolev space obtained by closing  the space of smooth functions with 
compact support in $\Omega$ in the norm $\left(\int_\Omega (|f|^2+|\nabla f|^2)dx\right)^{1/2}$. 
This form is regular on $\Omega$. The intrinsic distance is equal to the inner Euclidean distance in $\Omega$ (obtained by minimizing the length of the 
curves in $\Omega$ joining two points of $\Omega$, see the next section) and property {\em (A1)} 
is satisfied.  Property {\em (A2)} is not satisfied but {\em (A2-$Y$)} 
holds true for any $Y$ with $\overline {Y} \subset \Omega$. 
\end{example}

\subsection{Inner metric} \label{sec-innermetric}

Assume  (A1)-(A2) and let $\Omega$ be a non-empty domain in $X$.
For any continuous path $\gamma: [0,1]\rightarrow Y$, set
\[ \textrm{length}(\gamma) = \sup \left\{ \sum_{i=1}^n d(\gamma(t_i),\gamma(t_{i-1})) : n \in \mathbb{N}, 0 \leq t_0 < \ldots < t_n \leq 1 \right\}. \]

\begin{definition}
The \emph{inner metric} on $\Omega$ is defined as
 \[ d_{\Omega}(x,y) = \inf \big\{ \mbox{\em length}(\gamma) \big| \gamma:[0,1] \to \Omega \textrm{ continuous}, \gamma(0) = x, \gamma(1) = y \big\}. \]
Let $\widetilde\Omega$ be the completion of $\Omega$ with respect to $d_{\Omega}$.\end{definition} 
Whenever we consider an inner ball $B_{\widetilde\Omega}(x,R) = \{ y \in \widetilde\Omega : d_{\Omega}(x,y)<R \}$ or $B_{\Omega}(x,R) = B_{\widetilde\Omega}(x,R) \cap \Omega$, we assume that its radius is minimal in the sense that $B_{\widetilde\Omega}(x,R) \neq B_{\widetilde\Omega}(x,r)$ for all $r < R$. 
If $x$ is a point in $\Omega$, denote by $\delta(x) = \delta_{\Omega}(x) = d(x,\partial\Omega)$ the distance from $x$ to the boundary of $\Omega$. Let $\mbox{diam}_{\Omega}(\Omega)$ be the diameter of $\Omega$ in the inner metric $d_{\Omega}$.

\begin{definition} \label{F0UV}
For two open sets $V\subset \Omega$, let
\begin{align*} \F^0_{\mbox{\em \tiny{loc}}}(\Omega,V) 
=& \left\{ f \in L^2_{\mbox{\em \tiny{loc}}}(\Omega) : \forall\, W\subset V, \mbox{ rel. cpt. in } \widetilde{\Omega} \mbox{ with } d_\Omega(W,\Omega\setminus V)>0,\right.\\
&\left. \exists\, f^{\sharp} \in \F^0(U) \textrm{ such that } f=f^{\sharp} \textrm{ a.e. on } W \right\}. \end{align*}
\end{definition}

\begin{definition} Let $\Omega$ be a domain in $X$.
For an open set $V \subset \Omega$, let $V^{\sharp}$ be the largest open set in 
$\widetilde \Omega$ which is contained in the closure of $V$ in 
$\widetilde \Omega$ and 
whose intersection with $\Omega$ is $V$.
\end{definition}

\begin{lemma} \label{lem:2.46}
Let $V$ be an open set in $\Omega$. A function $g \in \F_{\mbox{\em{\tiny{loc}}}}(V)$ is in $\F^0_{\mbox{\em{\tiny{loc}}}}(\Omega,V)$ if and only if we have 
$f g \in \F^0(\Omega)$ for any bounded function $f \in \F(\Omega)$ 
with compact support in $V^{\sharp}$ and such that $d\Gamma(f,f)/d\mu \in L^{\infty}(\Omega,\mu)$. 
\end{lemma}

\begin{proof}
See \cite[Lemma 2.46]{GyryaSC}.
\end{proof}

\subsection{The doubling property and Poincar\'e inequality}

Let $Y \subset X$ be open 
and assume that the intrinsic metric $d$ satisfies (A1)-(A2)
(more generally, (A1) and (A2-$Y$) suffices). 
\begin{definition}
The form 
$(\e^{\mbox{\em\tiny s}},\F)$ 
satisfies the \emph{volume doubling property} on $Y$ if there exists a constant $D_Y \in (0,\infty)$ such that for every ball $B(x,2r) \subset Y$, 
\begin{align} 
V(x,2r) \leq D_Y \, V(x,r), \tag{VD}
\end{align}
where $V(x,r) = \mu( B(x,r))$ denotes the volume of $B(x,r)$.
\end{definition}

\begin{definition}
The form $(\e^{\mbox{\em \tiny s}},\F)$ 
satisfies the (weak) \emph{Poincar\'e inequality} on $Y$ if there exists a constant $P_Y \in (0,\infty)$ such that for any ball $B(x,2r) \subset Y$,
\begin{align}
\forall f \in D(\e), \  \int_{B(x,r)} |f - f_B|^2 d\mu  \leq  P_Y \, r^2 \int_{B(x,2r)} d\Gamma(f,f), \tag{PI}
\end{align}
where $f_B = \frac{1}{V(x,r)} \int_{B(x,r)} f d\mu$ is the mean of $f$ over $B(x,r)$. 
\end{definition}
The term weak refers to the fact that the ball $B(x,2r)$ is used on the 
right-hand side of the Poincar\'e inequality. 
It will be omitted in what follows. Under the doubling condition, 
strong and weak versions of the Poincar\'e inequality are in fact equivalent 
(e.g., \cite{SC02}).

If $Y=X$, the properties introduced in these definitions have a 
very global nature
as they hold uniformly at all scales and locations.  It is natural to introduce 
a more local version of these properties.

\begin{definition}
The form $(\e^{\mbox{\em\tiny s}},\F)$ 
satisfies the volume 
doubling property and the  Poincar\'e inequality {\em locally}
on $Y$ if for all $x\in Y$ there is a neighborhood $Y(x)$ of $x$ so 
that the volume doubling property and the Poincar\'e inequality hold in $Y(x)$.  

The form $(\e^{\mbox{\em\tiny s}},\F)$ 
satisfies the volume 
doubling property and the Poincar\'e inequality {\em up to scale} $R$ 
in $Y$ if 
the volume doubling property and the Poincar\'e inequality hold in $B(x,2R)$
with constants independent of $x$, for all $x\in Y$.  
\end{definition}

\begin{example} Let $(M,g)$ be a complete Riemannian manifold and $Y$ 
an open subset of $M$. Equip $M$ with its Riemannian measure and 
the Dirichlet form   $\e^{\mbox{\em\tiny s}}(f_1,f_2)=
\int_M g(\nabla f_1,\nabla f_2)d\mu$
with its natural domain $\F$ (the first Sobolev space on $M$). In this case, the intrinsic distance on $M$ equals the Riemannian distance.
\begin{itemize}
\item The volume doubling property and the Poincar\'e 
inequality hold locally on on $Y$.  
\item If $Ric \ge -\kappa g$ on the $2R$-neighborhood of  $Y$ for some fixed  
$\kappa>0$ and $R>0$ then
the volume doubling property and the Poincar\'e 
inequality hold up to scale $R$ on $Y$.
\item If $Ric \ge 0$ on $Y$ then the volume doubling property and the Poincar\'e 
inequality hold on $Y$. 
\end{itemize}
\end{example}

\begin{example} Let $G$ be a unimodular Lie group equipped with 
its Haar measure and with a family $\{X_1,\dots,X_k\}$ of left invariant 
vector fields that, viewed as elements of the Lie algebra, generates 
the Lie algebra of $G$ (this condition is often called the 
H\"ormander condition). Consider the Dirichlet form
$\e^{\mbox{\em\tiny s}}(f_1,f_2)=\int_G \sum_i X_if_1 X_if_2 d\mu$
with its natural domain $\F$, the space of functions in $L^2(G,\mu)$
such that, for each $i$, the distribution $X_if$ can be represented by an element of $L^2(G,\mu)$. In this case, the intrinsic distance is equal to the associated 
sub-Riemannian distance.
\begin{itemize}  
\item The volume doubling property and the Poincar\'e 
inequality hold up to scale $R$ on $G$ for any fixed $R>0$.
\item If $G$ has polynomial volume growth (i.e., $\exists \,A,\forall\,r>0,\;\;V(e,r)\le Cr^A$)
then the volume doubling property and the Poincar\'e 
inequality hold on $G$. 
\end{itemize}
See, e.g., \cite[Section 5.6]{SC02} and \cite{VSCC}.
\end{example}

\subsection{Carr\'e du champ and Lipschitz functions}


\begin{theorem} \label{thm:Upsilon}
Suppose the form $(\e^{\mbox{\emph{\tiny{s}}}},\F)$ satisfies \emph{(A1)-(A2)}, and the volume doubling property holds locally on $X$. Then for any Lipschitz function $f$ with Lipschitz constant $C_L$, the energy measure $d\Gamma(f,f)$ is absolutely continuous with respect to $d\mu$ and the Radon-Nikodym derivative $\Upsilon(f,f) = d\Gamma(f,f) / d\mu$ satisfies
 \[ \Upsilon(f,f) \leq C_L^2 \]
 almost everywhere.
\end{theorem}

\begin{proof}
See \cite[Theorem 2.1, Remark 2.1(ii)]{KosZhou}.
\end{proof}

The next corollary is used to prove Proposition \ref{prop:5.8 h} and Lemma \ref{lem:A1 for d_U}.

\begin{corollary} \label{cor:2.22} Let $\Omega$ be a domain in $X$. Suppose the model form $(\e^{\mbox{\emph{\tiny{s}}}},\F)$ satisfies \emph{(A1)-(A2-$\overline{\Omega}$)}, and the volume doubling property holds locally on $\overline{\Omega}$. Then any function $f$ on $\Omega$ which is Lipschitz with respect to $d_\Omega$ 
with Lipschitz constant $C_L$ is in $\F_{\mbox{\em{\tiny{loc}}}}(\Omega)$ 
and satisfies
 \[  C_L \geq \sup_{\Omega} \sqrt{ \Upsilon(f,f) }. \]
\end{corollary}

\begin{proof}
Follows from Theorem \ref{thm:Upsilon} and a simple adaption of the arguments in \cite[Corollary 3.6]{Hir03}, \cite{Wea00} or \cite[Corollary 2.22]{GyryaSC}.
\end{proof}

\section{Inner uniformity}  \label{ssec:uniform domains}
Let $X,\mu, \e^{\mbox{\tiny{s}}},\F, d$ be as above and assume that (A1)-(A2)
are satisfied so that $(X,d)$ is a complete metric space.

\subsection{Inner uniform domains}

\begin{definition}
Fix $c \in (0,1)$, $C \in (1,\infty)$. Let $\Omega$ be a domain in $X$.
Let $\gamma: [\alpha,\beta] \to \Omega$ be a rectifiable curve in $\Omega$. 
We say that $\gamma$ is a $(c,C)$-\emph{uniform} curve in $\Omega$ if 
the following two conditions are satisfied:
\begin{enumerate}
\item $\forall   t \in [\alpha,\beta]$, $\delta_{\Omega} \big( \gamma(t) \big) \geq c  \, \min\left\{ d \big( \gamma(\alpha), \gamma(t) \big) , d \big( \gamma(t), \gamma(\beta) \big) \right\}$
\item $ \textrm{\em length}(\gamma) \leq C \, d \big( \gamma(\alpha), \gamma(\beta) \big).$
\end{enumerate}
The domain $\Omega$ is called
$(c,C)$-\emph{uniform} if any two points in $\Omega$ can be joined by a $(c,C)$-uniform curve in $\Omega$.
\end{definition}

\begin{definition} \label{def-IU}
Fix $c \in (0,1)$, $C \in (1,\infty)$. 
\begin{enumerate}
\item Let $\gamma: [\alpha,\beta] \to \Omega$ be a rectifiable curve in 
$\Omega$. 
We say that $\gamma$ is a $(c,C)$-\emph{inner uniform} curve in $\Omega$ 
if its is $(c,C)$-uniform in $\Omega$  in  $(\widetilde{\Omega},d_\Omega)$. 
\item We say that the domain $\Omega$ is $(c,C)$-\emph{inner uniform} 
if $\Omega$ 
is $(c,C)$-uniform  in $(\widetilde{\Omega},d_\Omega)$.
\end{enumerate}
\end{definition}

\begin{remark}
The notions of $(c,C)$-length-uniformity and  inner-$(c,C)$-length-uniformity 
are defined analogously by replacing 
$d(\gamma(s),\gamma(t))$ by $\textrm{length} (\gamma\big|_{[s,t]})$ in condition 
{\em (i)}. The arguments used in    \cite[Lemma 2.7]{MS79} and 
\cite[Proposition 3.3]{GyryaSC} show that if $\gamma$ is a ($c$-$C$)-uniform curve
in $\Omega$ joining $x$ and $y$ of length  at most $R$ and if 
the doubling property holds in $B(x,2R)$ then there is a $(c',C')$-length 
uniform curve joining $x$ and $y$ in $\Omega$. For our purpose, this means that 
uniformity (resp. inner uniformity) and length-uniformity 
(resp. inner-length-uniformity) are equivalent notions.
\end{remark}

\begin{lemma} \label{lem:x_r} 
Let $\Omega$ be a $(c_u,C_u)$-inner uniform domain in $(X,d)$. 
For every ball $B = B_{\widetilde\Omega}(x,r)$ in $(\widetilde\Omega,d_{\Omega})$ with minimal radius, there exists a point $x_r \in B$ with $d_{\Omega}(x,x_r) = r/4$ and $d(x_r,\widetilde\Omega \setminus \Omega) \geq c_ur/8$.
\end{lemma}

\begin{proof} This is immediate, see \cite[Lemma 3.20]{GyryaSC}.
\end{proof}

Proving that a domain $\Omega$ is inner uniform is a difficult task. 
In fact, we lack a general method of constructing inner uniform domains 
in, say, complete metric length spaces. On the other hand, many domains 
in Euclidean space are inner uniform.

\begin{example} In Euclidean space, any bounded convex domain is uniform. In addition, if $\Omega$ is convex and  
$B(x,aR)\subset \Omega \subset B(x,AR)$ then the uniformity constants
$c_u,C_u$ depend only on $a,A$.  Any bounded 
domain with piecewise smooth boundary with
a finite number of singularities and non-zero interior angle at each of 
the singularities is inner uniform.  The open unit ball in $\mathbb R^n$, 
$n\ge 2$, with the trace of the half-hyperplane $\{x: x_n=0, x_{n-1}<0 \}$ 
deleted is inner uniform. The interior and exterior of the Koch snowflake 
are inner uniform domains (in fact, uniform).  
The exterior of any convex set is inner uniform.
\end{example} 

\begin{example} Let $G=\mathbb R^3$ be the Heisenberg group with law 
$$g_1g_2= (x_1+x_2,y_1+y_2, z_1+z_2+(1/2)(x_1y_2-x_2y_1)), 
\;\;g_i=(x_i,y_i,z_i).$$
Let $X$ and $Y$ be the left invariant vector fields on $G$ with $X(0)=\partial_x$, $Y(0)=\partial_y$. Let $\e(f,f)=\int_G( |Xf|^2+|Yf|^2)d\mu$ where $\mu$ denotes 
the Haar measure on $G$ and the domain of $\e$ is the closure of smooth compactly functions for the norm $(\int (|f|^2+ |Xf|^2+|Yf|^2)d\mu)^{1/2}$. Let $d$ 
be the corresponding intrinsic distance.   Examples of uniform domains include
any coordinate half-space through the origin, the coordinate unit cube in 
$\mathbb R^3$ and any metric ball $B(x,r)$ in $(G,d)$. See 
\cite{Greshnov1,Greshnov2} and \cite{GyryaSC} for further pointers 
to the literature.
\end{example}

\subsection{Local inner uniformity} \label{ssec:local inner uniformity}

In \cite{LierlSC1}, the authors derived a scale invariant boundary 
Harnack principle  under a local version of inner uniformity which we now recall.


\begin{definition} \label{def:R_xi}  
Fix $c_u \in (0,1)$, $C_u \in (1,\infty)$ and a domain $\Omega$.
For a point $\xi \in \widetilde{\Omega}$,
 let $R(\Omega,\xi) \in [0,\infty]$ be the 
largest  $R\ge 0$ so that
\begin{enumerate}
\item
$ 8 R \leq  \textrm{\em diam}_{\Omega}(\Omega)$ 
(this is a non-trivial condition only when $\Omega$ is a bounded domain),
\item
Any two points in $B_{\widetilde\Omega}(\xi, 8R)$ 
can be connected by a curve that is $(c_u,C_u)$-inner uniform in $\Omega$. 
\end{enumerate}
\end{definition}

\begin{remark} \label{rem-R>}
It easily follows from Definition \ref{def:R_xi} that if $\xi$ 
is such that $R(\Omega,\xi)>0$ then there exists 
$\eta>0$ such that  
$$d_\Omega(\xi,\xi')< R(\Omega,\xi) \Longrightarrow 
R(\Omega,\xi')>\eta R(\Omega, \xi).$$
\end{remark}

Consider non-empty domains $W \subset \Omega\subset X$. 
Let $W^{\sharp}$ be the largest open set 
in $(\widetilde \Omega, d_\Omega)$ whose intersection with $\Omega$ is $W$.

\begin{remark}
Any inner ball in $(\widetilde W,d_W)$ that lies in $W^{\sharp}$ 
is also an inner ball in $(\widetilde \Omega,d_\Omega)$. 
However, the metrics $d_W$ and $d_\Omega$ do not necessarily coincide.
\end{remark}

\begin{definition} Fix non-empty  domains $W\subset \Omega\subset X$.
\begin{enumerate}
\item We say that $\Omega$ is locally inner uniform near $W$ if for any point 
$\xi\in W^{\sharp}$ 
we have  $R(\Omega,\xi)>0$ . 
\item We say that $\Omega$ is locally inner uniform up to scale $R>0$
near $W$ if for any point 
$\xi\in W^{\sharp}$, we have
$R(\Omega,\xi)\ge  R$.  
\end{enumerate}
\end{definition}

\begin{remark} 
\begin{enumerate}
\item 
From these definitions, it follows easily that if $\Omega$ is 
$(c_u,C_u)$-inner uniform 
then  $R(\Omega,\xi) 
\simeq \mbox{diam}_\Omega(\Omega)$ 
for each 
$\xi \in  \widetilde{\Omega}$. 
The constants implicitly contained
in the notation $\simeq$ depend only on $c_u,C_u$.
\item By Remark \ref{rem-R>}, if $\Omega$ is locally inner uniform near $W$
and $\xi\in W^{\sharp}$, then there exists $R_\xi$ such that
$\Omega$ is  locally inner uniform up to scale $R_\xi$ 
near $B_{\Omega}(\xi, R_\xi)$.    
\item Assume that $\Omega$ is locally 
$(c_u,C_u)$-inner uniform up to scale $R$ near $W$. Then for any point 
$\xi\in W^{\sharp}$ and $r\in (0,R)$ there exists a point $\xi_r\in \Omega$
such that $d_\Omega(\xi,\xi_r)=r/4$ and 
$d(\xi_r,\widetilde{\Omega}\setminus \Omega)\ge c_u r/8$.  
See Lemma \ref{lem:x_r}  and \cite[Lemma 3.20]{GyryaSC}.
\item In $\mathbb R^n$, any domain with smooth boundary 
is locally inner uniform. Many such domains 
(e.g., an unbounded ``turnip'' domain) are not locally 
inner uniform up to scale $R$.
\end{enumerate}
\end{remark}

\section{Adapted forms} \label{sec3:non-sym forms}

In this section, we introduce a large class of real bilinear  
forms  on $L^2(X,d\mu)$ that all share a common domain $\F$, the domain of
our model form $\e^{\mbox{\tiny s}}$. Further, these forms are of the type
$\e^{\mbox{\tiny s}}+\mbox{ lower order terms}$. 
Our goal is to pick one of these forms, $\e$, and to study the 
Dirichlet heat kernel (and Dirichlet semigroup) associated to $\e$ 
in a domain $U$ under the hypothesis that $U$ 
is inner uniform or, more generally, locally inner uniform.

\subsection{First and zero order parts}
 
Given a bilinear form $\e$, we set
$$\e^{\mbox{\tiny sym}}(f,g)=\frac{1}{2}(\e(f,g)+\e(g,f)) \;\mbox{ and }\;\;
\e^{\mbox{\tiny skew}}(f,g)=\frac{1}{2}(\e(f,g)-\e(g,f)).$$
These are, respectively,  the symmetric and skew part of $\e$.
For $f,g\in \F_{\mbox{\tiny{c}}}\cap L^\infty(X,\mu)$, we also set
$$\l(f,g)= \frac{1}{2}\left(\e^{\mbox{\tiny skew}}(fg,1)+
\e^{\mbox{\tiny skew}}(f,g)\right)\;\mbox{ and } \;\; \r(f,g)=-\l (g,f) .$$
Obviously, 
$$\e^{\mbox{\tiny skew}}(f,g)=\l(f,g)+\r(f,g).$$

We recall the following definition taken from \cite{LierlSC2}.
\begin{definition} \label{def-firstorder}
Assuming $\e$ is local with $D(\e)=\F$, 
we say that  $\e^{\mbox{\em\tiny{skew}}}$ is a 
{\em chain rule skew form relative to $\F$} 
if the following two properties hold:
\begin{itemize}
\item
For any $u,v,f \in \F\cap \mathcal C_{\mbox{\em\tiny c}}(X)$, we have
 \[ \l(uf,v) = \l(u,fv) + \l(f,uv). \]
\item Let $v,u_1, u_2, \ldots, u_m \in 
\F\cap \mathcal C_{\mbox{\em\tiny c}}(X)$ and $u = (u_1, \ldots, u_m)$. If $\Phi \in C^2(\R^m)$, then $\Phi(u), \Phi_{x_i}(u) \in \F_{\mbox{\em{\tiny{loc}}}}(X)\cap L^\infty_{\mbox{\em \tiny loc}}(X,\mu)$ and
\begin{align*}
\l(\Phi(u),v) = \sum_{i=1}^{m} \l(u_i, \Phi_{x_i}(u) v).
\end{align*}
\end{itemize}
\end{definition}

\begin{definition} \label{def-adapt}
We say that the form $(\e,D(\e))$ is adapted
to $(\e^{\mbox{\em \tiny s}},\F)$ if $\e$ is local,
its domain $D(\e)$ is $\F$ and:
\begin{enumerate}
\item The form $\e$ satisfies 
$$\forall\,f,g\in \mathcal F,\;\;|\e(f,g)|\le C\|f\|_\F\|g\|_\F,$$
and, for all $f,g\in \F$ with $fg\in \F_{\mbox{\em\tiny{c}}}$,
$$|\e(fg,1)|+|\e(1,fg)|\le C\|f\|_\F\|g\|_\F.$$
\item The symmetric bilinear form 
$\e^{\mbox{\em\tiny sym}}(f,g)-
\e^{\mbox{\em\tiny sym}}(fg,1)$, extended by continuity to $\F$, 
is equal to the model form $\e^{\mbox{\em\tiny s}}$. 
\item The skew part $\e^{\mbox{\em\tiny{skew}}}$ 
is a chain rule skew form relative 
to $\F$. 
 \end{enumerate}
\end{definition}

\begin{definition} A symmetric bilinear form $Z$ is said to be a zero order form
adapted to $\F$ if it is defined on $\F$ and satisfies 
$$Z(f,g)=Z(fg,1), \quad f,g\in \F, fg\in \F_{\mbox{\em\tiny c}},$$
and 
$$|Z(f,g)|\le C \|f\|_\F\|g\|_\F.$$
\end{definition}
Since $(\e^{\mbox{\tiny s}},\F)$ is fixed throughout, we will simply say 
that $(\e,D(\e))$ is an adapted form and 
that $Z$ is an adapted symmetric zero order form.  Note that if $\e$ is
an adapted form then its symmetric zero order part $Z_\e(f,g)=
\e^{\mbox{\tiny sym}}(fg,1)$ is a zero order form adapted to $\F$.
Further, $\e= \e^{\mbox{\tiny s}}+\e^{\mbox{\tiny skew}}+Z_\e$.

\subsection{Quantitative assumptions on the forms}
We now introduce the fundamental quantitative assumptions on the 
bilinear forms for which we will study weak solutions of the heat equation
with Dirichlet boundary condition.   

\begin{assumptionA} \label{as3:e_t} 
The form  $(\e,D(\e))$ is a bilinear form on $L^2(X,\mu)$ which is adapted to 
the model form $(\e^{\mbox{\em\tiny s}},\F)$.
Let $C_0=C_0(\e)$ 
be the constant in 
the sector condition 
$|\e^{\mbox{\em\tiny skew}}(f,g)|\le C_0 \|f\|_\F\|g\|_\F$. 
Assume further that:
\begin{enumerate}
\item
There are constants $C_2(\e), C_3(\e) \in [0,\infty)$ 
so that for all $f \in \F$ 
with $f^2 \in \F_{\mbox{\emph{\tiny{c}}}}$,
\begin{align} \label{ZEA}
\left| \e^{\mbox{\emph{\tiny{sym}}}}(f^2,1) \right|
\leq 2 \left(\int f^2 d\mu \right)^{\frac{1}{2}} \left(C_2(\e) \int d\Gamma(f,f) + C_3(\e) \int f^2 d\mu \right)^{\frac{1}{2}}
\end{align}
\item
There is a constant $C_5(\e) \in [0,\infty)$ such that for all $f \in \F$, 
$g \in \F_{\mbox{\emph{\tiny{c}}}} \cap L^{\infty}(X)$,
\begin{align} \label{skewA}
 \left| \e^{\mbox{\emph{\tiny{skew}}}}(f,fg^2) \right|
\leq & 2 \left( \int f^2 d\Gamma(g,g) \right)^{\frac{1}{2}} \left( C_5(\e) \int f^2 g^2 d\mu \right)^{\frac{1}{2}}.
\end{align}
\end{enumerate}
Set 
 \[ C_8(\e) := C_2(\e)+ C_3(\e)^{1/2}+ 
C_5(\e). \]
\end{assumptionA}

\begin{remark}  Under Assumption A
the form $(\e,\F)$ is closed and 
satisfies $$\forall\,f\in \F,\;\;\;\e(f,f)\ge -\alpha\|f\|_2^2,$$ with $\alpha$ 
depending only on $C_2(\e),C_3(\e)$.  
In particular, the form $(\e,\F)$ induces a continuous semigroup of 
bounded operators $P_t$ on $L^2(X,\mu)$. 
We let $(L,D(L))$ denote the infinitesimal 
generator of this semigroup. By the results of \cite{MaRock}, it is immediate 
that $P_t$ is positivity preserving.
\end{remark}
\begin{remark}
 For the purpose of this work, it is essential to compare 
Assumption A to Assumptions 0-1-2 of \cite{LierlSC2}. 
\begin{enumerate}
\item It is plain that any form $\e$ 
satisfying Assumption A also satisfy 
Assumptions 0-1-2 of \cite{LierlSC2} with respect to the model form 
$(\e^{\mbox{\tiny s}},\F)$.  Regarding Assumption 2 of \cite{LierlSC2}, see 
\cite[Remark 1.15(iv)]{LierlSC2}.
\item  Given a model form $(\e^{\mbox{\tiny s}},\F)$, 
forms satisfying Assumption A are less general than
the forms allowed by Assumptions 0-1-2 of \cite{LierlSC2}. 
To understand this, compare Assumption A(ii) with  
\cite[Assumption 1(iii)]{LierlSC2} and note that   
Assumption A(ii)  is the same as 
\cite[Assumption 1(iii)]{LierlSC2} with $C_4=0$.   
\end{enumerate}
\end{remark}

\begin{remark}
On Euclidean space, fix measurable bounded functions $a_{i,j}$, $b_i$, $d_i$, $c$,
set  $\F=D(\e)=W^1(\R^n)$ and
\begin{align*}
 \e(f,g) & = \int \! \sum_{i,j=1}^n  a_{i,j} \partial_i f \partial_j g \, dx
           + \int \sum_{i=1}^n b_i \partial_i f \, g \, dx
           + \! \int \sum_{i=1}^n f \, d_i \partial_i g \, dx
           + \int \! c f g \, dx.
\end{align*}
Set $\tilde a_{i,j} := (a_{i,j} + a_{j,i})/2$ and $\check a_{i,j} = (a_{i,j} - a_{j,i})/2$. Then the symmetric part of $\e$ is 
\begin{align*}
 \e^{\mbox{\tiny{sym}}}(f,g) 
 = & \int \sum_{i,j=1}^n \tilde a_{i,j} \partial_i f \partial_j g \, dx
           + \int \sum_{i=1}^n \frac{b_i+d_i}{2} \partial_i f \, g \, dx \\
   &        + \int \sum_{i=1}^n f \, \frac{b_i+d_i}{2} \partial_i g \, dx
       + \int c f g \, dx,
\end{align*}
while the skew-symmetric part of $\e$ is 
\begin{align*}
 \e^{\mbox{\tiny{skew}}}(f,g) 
 = & \int \sum_{i,j=1}^n \check a_{i,j} \partial_i f \partial_j g \, dx
           + \int \sum_{i=1}^n \frac{b_i-d_i}{2} \partial_i f \, g \, dx \\
   &        + \int \sum_{i=1}^n f \, \frac{-b_i + d_i}{2} \partial_i g \, dx.
\end{align*}
The symmetric part $\e^{\mbox{\tiny{sym}}}$ can be decomposed into its strongly local part
\begin{align*}
 \e^{\mbox{\tiny{s}}}(f,g) & = \sum_{i,j=1}^n \int \tilde a_{i,j} \partial_i f \partial_j g \, dx
\end{align*}
and its symmetric zero order part given by 
$$
\e^{\mbox{\tiny sym}}(fg,1) = 
\int \sum_{i=1}^n \frac{b_i+d_i}{2} \partial_i (fg)\, dx 
       + \int c f g \, dx .$$
Assume that $(\tilde{a}_{i,j})$ is uniformly elliptic 
and set 
$$\e^{\mbox{\tiny s}}(f,g)=
\int \sum_{i,j=1}^n \tilde a_{i,j} \partial_i f \partial_j g \,dx,\;\; f,g \in 
\F.$$
On the one hand, under these hypotheses, the form $\e$ satisfies  
\cite[Assumptions 0-1-2]{LierlSC2}. On the other hand, making the hypothesis  
that $\e$ is an adapted  form with  respect to $(\e^{\mbox{\tiny s}},\F)$ 
implies that the matrix $(a_{i,j})$ is symmetric, i.e., $(a_{i,j})=(\tilde{a}_{i,j})$.

Further, under these circumstances, the constants $C_2(\e),C_5(\e)$ 
can be taken  to be equal to $0$ if 
$b_i=d_i=0$ for all $i$ (i.e., if there is no drift term). 
The constant $C_8(\e)$ can be taken equal to $0$ if $b_i=d_i=c=0$. 
\end{remark} 

We will need the following simple Caccioppoli-type lemma. The proof is omitted.
\begin{lemma}\label{caccio} Let $(\e,\F)$ be a form satisfying 
\emph{Assumption A}. 
Let $u\in \F_{\mbox{\em\tiny loc}}$ and $\psi\in 
\F_{\mbox{\em\tiny c}} \cap L^\infty(X,\mu)$.
For any $k_1>0$, we have
$$-\e^{\mbox{\em\tiny s}}(u,u\psi^2)\le 4k_1\int u^2d\Gamma(\psi,\psi)-
\left(1-\frac{1}{k_1}\right)\int \psi^2d\Gamma(u,u).$$
Moreover, for any $k_1,k_2,k_3>0$,
\begin{eqnarray*}
-\e(u,u\psi^2)&\le& (4k_1+2k_2C_2+k_3)\int u^2 d\Gamma(\psi,\psi)\\
&&+\left(-1+\frac{1}{k_1}+2k_2C_2\right)\int\psi^2d\Gamma(u,u)\\
&& +\left(\frac{1}{k_2}+k_2C_3+\frac{C_5}{k_3}\right)\int u^2\psi^2d\mu.
\end{eqnarray*}
\end{lemma}

\subsection{Local weak solutions}
Consider an adapted form $(\e,\F)$. 
Let $V $ be an open set. Recall that
\begin{align*}
 \F(V)   &=  \{ u \in \F_{\mbox{\tiny{loc}}}(V) : \int_V |u|^2 d\mu + \int_V d\Gamma(u,u) < \infty \}.
\end{align*}

\begin{definition}
Let $V $ be open and $f \in \F_{\mbox{\em\tiny{c}}}(V)'$, 
the dual space of $ \F_{\mbox{\em\tiny{c}}}(V)$ (identify $L^2(X,\mu)$ with its dual space using the scalar product). A function $u: V \to \R$ is a 
\emph{local weak solution} of the Laplace equation $-Lu = f$ in $V$, if 
\begin{enumerate}
\item $u \in \F_{\mbox{\em \tiny{loc}}}(V)$,
\item  For any function $\phi \in \F_{\mbox{\em \tiny{c}}}(V), \ \e(u,\phi) = \int f \phi \, d\mu$.
\end{enumerate} 
\end{definition}

For a time interval $I$ and a Hilbert space $H$, let $L^2(I \to H)$ be the Hilbert space of those functions $v: I \to H$ such that 
 \[  \Vert v \Vert_{L^2(I \to H)} = \left( \int_I \Vert v(t) \Vert_H^2 \, dt \right)^{1/2} < \infty. \]
Let $\mathcal W^1(I \to H) \subset L^2(I \to H)$ be the Hilbert space of those functions $v: I \to H$ in $L^2(I \to H)$ whose distributional time derivative $v'$ can be represented by functions in $L^2(I \to H)$, equipped with the norm
 \[  \Vert v \Vert_{\mathcal W^1(I \to H)} = \left( \int_I \Vert v(t) \Vert_H^2 + \Vert v'(t) \Vert_H^2 \, dt \right)^{1/2} < \infty. \]
Let
 \[ \F(I \times X) = L^2(I \to \F) \cap \mathcal W^1(I \to \F'), \]
where $\F'$ denotes the dual space of $\F$.
Let
 \[ \F_{\mbox{\tiny{loc}}}(I \times V) \] 
be the set of all functions $u:I \times V \to \R$ such that for any open interval $J$ that is relatively compact in $I$, and any open subset $A$ relatively compact in $V$, there exists a function $u^{\sharp} \in \F(I \times X)$ such that $u^{\sharp} = u$ a.e. in $J \times A$.
Let 
\[ \F_{\mbox{\tiny{c}}}(I \times V) = \{ u \in \F_{\mbox{\tiny{loc}}}(I \times V): u \textrm{ has compact support in } I\times V  \}. \]

\begin{definition}
Let $I$ be an open interval and $V$ an open set in $X$. 
Set $Q = I \times V$. A function $u: Q \to \R$ is a \emph{local weak solution} of the heat equation $\frac{\partial}{\partial t} u = L u$ in $Q$, if
\begin{enumerate}
\item
$u \in \F_{\mbox{\emph{\tiny{loc}}}}(Q)$,
\item 
For any open interval $J$ relatively compact in $I$,
\begin{align} \label{eq:loc weak sol}
 \forall \phi \in \F_{\mbox{\emph{\tiny{c}}}}(Q),\  \int_J \int_V \frac{\partial}{\partial t} u \, \phi \, d\mu \, dt + \int_J \e(u(t,\cdot),\phi(t,\cdot)) dt = 0.
\end{align}
\end{enumerate} 
\end{definition}

\begin{remark}
Assuming that the intrinsic distance satisfies (A1)-(A2), 
an equivalent definition of a local weak solution of $\frac{\partial}{\partial t} u = L u$ on $Q = I \times V$ is
\begin{enumerate}
\item
$u \in L^2(I \to \F )$,
\item
For any open interval $J$ relatively compact in $I$,
\begin{align*}
 -\int_J \int_V \frac{\partial}{\partial t} \phi \, u \, d\mu \, dt + \int_J \e(u(t,\cdot),\phi(t,\cdot)) dt = 0,
\end{align*}
for all $\phi \in \F(Q)$ with compact support in $J \times V$.
\end{enumerate}
See \cite{ESC}. The argument uses the existence of good cut-off functions
provided by (A1)-(A2).
\end{remark}

\subsection{Local weak solutions with Dirichlet boundary condition 
along $\partial U$}

To define weak solutions with Dirichlet boundary condition, we 
use Definition \ref{F0UV} where 
the space $\F^0_{\mbox{\tiny loc}}(U,V)$ is introduced.

\begin{definition}
Let $V,U$ be open with $V \subset U$. A function $u: V \to \R$ is a 
\emph{local weak solution} of the Laplace equation $-Lu = f$ in $V$
with 
\emph{Dirichlet boundary condition} along $\partial U$ if 
\begin{enumerate}
\item $u$ is a local weak solution of $-Lu=f$ in $V$ and
\item 
$u \in \F^0_{\mbox{\em \tiny{loc}}}(U,V)$.
\end{enumerate} 
\end{definition}

Next we fix an open interval $I$ and an open set $V$ in a domain $U$ in $X$
and define the notion of a local weak solution in $I\times V$
with Dirichlet boundary condition along the boundary of $U$.  
Recall that $\F^0(U)$ is the closure of $\F_{\mbox{\tiny c}}(U)$ for the norm
$(\int_U|f|^2d\mu+\int_Ud\Gamma(f,f))^{1/2}$. Define
\[\F^0(I\times U)= L^2(I \to \F^0(U)) \cap \mathcal W^1(I \to (\F^0(U))'). \] 
For $Q=I\times V$, define $\F^0_{\mbox{\tiny loc}}(U,Q)$ to be the set of all 
functions $v: Q\rightarrow \mathbb R$ such that, for any open interval 
$J\subset I$ relatively compact in $I$ and any open subset $W\subset V$ 
relatively compact in 
$\widetilde{U}$ with $d_U(W, U \setminus V)>0$, there exists a function 
$v^{\sharp}$ in $\F^0(I\times U)$ such that $u^{\sharp}=u$ a.e. in $J\times W$.

\begin{definition}
Let $I$ be an open interval and $V$ an open set in $X$. 
Set $Q = I \times V$. We say that a function $u: Q \to \R$ is a \emph{local weak solution} of the heat equation $\frac{\partial}{\partial t} u = L u$ in $Q$ 
with Dirichlet boundary condition along $\partial U$ if  
\begin{enumerate}
\item $u$ is a local weak solution of the heat equation in $Q$  and
\item 
$u \in \F^0_{\mbox{\emph{\tiny{loc}}}}(U,Q)$.
\end{enumerate}
\end{definition}

\section{Harnack inequalities}

Harnack inequalities play an essential and central role in the results 
obtained in this paper. The next two subsections discuss interior Harnack 
inequalities and boundary Harnack inequalities, respectively. 

In this section, 
we consider a fixed  open subset $Y$ of $X$. 
We assume that the model form $(\e^{\mbox{\tiny{s}}},\F)$, defined in Section \ref{sec-geo}, satisfies (A1)-(A2-$Y$).

\subsection{Interior Harnack inequalities}
For any $s \in \R$, $\tau > 0$, $\delta \in (0,1)$ and $B(x,2r) \subset Y$, 
define
\begin{align*}
 I &= \big(s - \tau r^2, s \big) \\
 B &= B(x,r) \\
 Q &= I \times B \\
 Q_-  &= \big( s - (3+\delta)\tau r^2/4, s - (3-\delta)\tau r^2/4 \big) \times \delta B \\
 Q_+  &= \big( s - (1+\delta)\tau r^2/4, s \big) \times \delta B.
\end{align*}

\begin{definition} \label{def:HI}
Let  $(\e,\F)$  be an adapted form. 
\begin{itemize}
\item We say that
$(\e,\F)$ satisfies the \emph{parabolic Harnack inequality} on $Y$ if, 
for any $\tau > 0$, $\delta \in (0,1)$, there exists a constant 
$H_Y(\tau, \delta) \in (0,\infty)$ such that, for any ball $B(x,2r) \subset Y$, any $s \in \R$, and any positive local weak solution $u$ of the heat equation
$\frac{\partial}{\partial t}u = Lu$ in $Q$, 
the following inequality holds.
\begin{align} \label{eq:parabolic HI}
\sup_{z \in Q_-} u(z)  \leq  H_Y \inf_{z \in Q_+} u(z) \tag{PHI}
\end{align}
Here both the supremum and the infimum are essential, i.e.,
computed up to sets of measure zero.
\item We say that the parabolic Harnack inequality holds locally in $Y$ 
if for each $y\in Y$ there is a neighborhood $V$ of $y$ in $Y$ such that 
{\em (PHI)} holds in $V$ (in this case, the constant $H_V$ may indeed depend on $V$). 
\item We say that the parabolic Harnack inequality holds up to scale $R$ in $Y$ 
there is a constant $H_Y(R)$ such that 
{\em (PHI)} holds in any ball $B(y,2R)$, $y\in Y$, 
with  constant $H_{B(y,2R)}$ bounded above by $H_Y(R)$ . 
\end{itemize}
\end{definition}
The parabolic Harnack inequality implies the \emph{elliptic Harnack inequality},
\begin{align}
\sup_{z \in B(x,r)} u(z) \leq H'_Y \inf_{z \in B(x,r)} u(z), \tag{EHI}
\end{align}
where $u$ is any positive function in $\F_{\mbox{\tiny{loc}}}(Q)$ with 
$Lu = 0$ weakly in $B(x,2r)$.
Recall also that (PHI) implies the H\"older continuity of local weak solutions.

The following theorem gathers fundamental known results regarding 
the parabolic Harnack inequality. 

\begin{theorem} 
Let $X,Y,\e^{\mbox{\em \tiny s}},\F,d,\mu$ be as in {\em Section \ref{sec-geo}}.
In particular, we assume that {\em (A1)-(A2-Y)} holds true.
Let $(\e,\F)$ be a form satisfying \emph{Assumption A}.
\begin{enumerate}
\item  The symmetric strongly local regular
Dirichlet form $(\e^{\mbox{\em \tiny s}},\F)$
satisfies {\em (PHI)} on $Y$ if and only if 
it satisfies  the volume doubling property and the Poincar\'e inequality on $Y$.
\item  The symmetric strongly local regular 
Dirichlet form $(\e^{\mbox{\em \tiny s}},\F)$
satisfies {\em (PHI)} locally (resp.\ up to scale $R$) on $Y$ if and only if 
it satisfies  the volume doubling property and the Poincar\'e inequality locally 
(resp.\ up to scale $R$) on $Y$.
\item If the model form 
$(\e^{\mbox{\em \tiny s}},\F)$ satisfies {\em (PHI)} locally in $Y$ then 
the form $(\e,\F)$ satisfies {\em (PHI)} locally in $Y$.
\item If the model form 
$(\e^{\mbox{\em \tiny s}},\F)$ satisfies {\em (PHI)} locally up 
to scale $R<\infty$ 
in $Y$ with constant $H(\e^{\mbox{\em\tiny s}},R)$ 
then $(\e,\F)$ satisfies {\em (PHI)} up to scale $R<\infty$  in $Y$ with constant $H(\e,R)$ depending only on $H(\e^{\mbox{\em\tiny s}},R)$, the constants
$C_1(\e)$--$C_5(\e)$ and an upper bound on 
$C_8(\e)R^2$.

\end{enumerate}
\end{theorem}
\begin{remark}
The first two  statements of this theorem are the Dirichlet form version of the 
characterization of the parabolic Harnack inequality by  volume doubling 
and Poincar\'e inequality. See \cite{Gri91,SC92,SturmI,SturmII,SturmIII}.

Statements (iii)-(iv) are variations on the key fact that 
the parabolic Harnack inequality for the model form 
$(\e^{\mbox{\tiny s}},\F)$ implies  (PHI) for a wide variety 
of other forms in the spirit of the original work of Nash, Moser 
and Aronson and Serrin.  The proof is contained in 
\cite{LierlSC2,SturmII,SturmIII}.  In particular, (iii)-(iv) are special cases of  \cite[Theorem 2.13]{LierlSC2} which covers a wider class of 
forms, namely, forms satisfying \cite[Assumptions 0-1-2]{LierlSC2}.  
\end{remark}

\subsection{Boundary Harnack principle}

Let $(\e,\F)$ be an adapted form satisfying Assumption A.
Let $U$ be a domain in $X$. 
The boundary Harnack principle is concerned with positive 
local weak solutions of 
$Lu=0$ with Dirichlet boundary condition along $\partial U$ 
and their behavior near the boundary. We refer the reader to \cite{Aik}
for pointers to the literature. 


We will use a  strong version of the boundary Harnack principle
which we refer to as
the geometric boundary Harnack principle.

\begin{definition}
Let $X,\e^{\mbox{\em \tiny s}},\F,d,\mu$ be as in {\em Section \ref{sec-geo}}.
Let $ W\subset U$ be non-empty domains in $X$. 
Let  $(\e,\F)$  be a form satisfying 
{\em Assumption A}. Referring to 
local weak solutions of $Lu=0$ with Dirichlet boundary condition along 
$\partial U$ where $L$ is the generator associated to $(\e,\F)$, we say that: 
\begin{enumerate}
\item 
the \emph{geometric boundary Harnack principle} holds on $U$, 
if there exist constants $a_0,A_0,A_1\in (0,\infty)$, depending only on 
$U$, with the following property. Let $\xi \in \widetilde{U}\setminus U$ 
and $r \in (0,a_0 \mbox{ \em diam}_U (U))$. Then
for any two positive weak solutions $u$ and $v$ of $Lu=0$
in $B_U(\xi,A_0r)$ with Dirichlet boundary 
condition along $\partial U$, we have
 \[  \frac{u(x)}{u(x')} \leq A_1 \frac{v(x)}{v(x')}, \quad \forall x, x' \in 
B_U(\xi,r). \]
\item
the \emph{geometric boundary Harnack principle} holds locally 
near  $W$ if, for every compact set $K\subset W^{\sharp}\setminus W$, 
there exist 
$A_0(K),A_1(K)$ and $R(K)>0$ such that for any $\xi \in K$, $r \in (0,R(K))$  
and any two positive 
weak solutions $u$ and $v$ of $Lu=0$
in $B_U(\xi,A_0(K)r)$ with Dirichlet boundary 
condition along $\partial U$, we have
 \[  \frac{u(x)}{u(x')} \leq A_1(K) \frac{v(x)}{v(x')}, \quad \forall x, x' \in 
B_U(\xi,r). \]
\item the \emph{geometric boundary Harnack principle} holds up to 
scale $R$ near $W$ if we can take $A_0(K)=A_0$, $A_1(K)=A_1$ and $R(K)=R$ in the 
previous statement.
\end{enumerate}
\end{definition}

The following theorem follows immediately from  
\cite[Theorem 4.2]{LierlSC1}. 

\begin{theorem} \label{thm2:bHP for u}
Fix $R>0$.
Let $X,\e^{\mbox{\em \tiny s}},\F,d,\mu$ be as in {\em Section \ref{sec-geo}}.
Let $(\e,\F)$ be a form  satisfying {\em Assumption A}.
Let $W\subset U$ be  domains in $X$. Assume further that:
\begin{enumerate}
\item $(\e,\F)$ is a (possibly non-symmetric) Dirichlet form. 
\item The volume doubling property and the Poincar\'e 
inequality hold up to scale $R$ in $\overline{W}$.  
\item The domain $U$ is locally  $(c_u,C_u)$-inner uniform up to scale $R$ 
near $W$.
\end{enumerate}
Then there exist constants $a_0\in (0,1)$, $A_0,A_1 \in (1,\infty)$ 
such that for any 
$\xi \in W^{\sharp}\setminus W$, $0<r<a_0R$, 
and any two non-negative local weak solutions $u,v$ of $ L u = 0$ in 
$B_{U}(\xi,A_0r)$ with weak Dirichlet boundary condition along
$\partial U$, we have
\[  \frac{u(x)}{u(x')} \leq A_1 \frac{v(x)}{v(x')}, \]
for all $x, x' \in B_{U}(\xi,r)$. 

The constants $a_0,A_0$ depend only on the local 
inner uniformity constants $c_u,C_u$ near $W$.
The constant $A_1$ depends only on the inner uniformity constants $c_u, C_u$, 
an upper bound on the
volume doubling constant and the Poincar\'e inequality constant up to scale $R$ 
on $\overline{W}$, 
the constants $C_0(\e)$--$C_5(\e)$ from {\em Assumption A} which give control over the skew-symmetric part and the killing part of the Dirichlet form $\e$, 
and an upper bound on $C_8(\e) R^2$, 
\end{theorem}
The following theorem is a direct consequence of Theorem \ref{thm2:bHP for u}
and the various definitions. 
\begin{theorem} \label{th-BHP1}
Let $X,\e^{\mbox{\em \tiny s}},\F,d,\mu$ be as in {\em Section \ref{sec-geo}}.
Let $(\e,\F)$ be a form satisfying {\em Assumption A}.    
Assume further that $(\e,\F)$ is a Dirichlet form. 
Let $U$ be a domain in $X$.    
\begin{enumerate}
\item Fix a domain $W\subset U$, 
and assume that $U$ is locally inner uniform near $W$. Assume 
also that the volume doubling property and the Poincar\'e inequality hold locally in $\overline{W}$. 
Then the geometric  boundary Harnack principle holds locally in $U$ near $W$.
\item  Fix $R\in (0,\infty]$ and a domain $W\subset U$.  
Assume that $U$ is locally $(c_u,C_u)$-inner uniform up to scale $R$ near $W$
and that the volume doubling
property and Poincar\'e inequality hold up to scale $R$ in $\overline{W}$. 
Then there exists $a_0>0$  such that 
the geometric boundary Harnack principle holds locally up to scale 
$r$ near $W$ for all $r< a_0R$ with $C_8(\e)r^2<\infty$.
\item  Assume that $U$ is inner uniform 
and that the volume doubling
property and Poincar\'e inequality hold in $X$. Assume further that 
$\e=\e^{\mbox{\em \tiny s}}$. Then the
the geometric boundary Harnack principle holds true in $U$. 
\end{enumerate}
\end{theorem}

\subsection{Harmonic profiles}
The main idea developed in \cite{GyryaSC} in the context of strongly local 
symmetric Dirichlet forms is that the Dirichlet heat kernel
in a domain $U$ can be estimated in terms of the harmonic profile $h_U$.
In this section we extend the notion of harmonic profiles 
and gather some of 
their key properties. 

\begin{definition} \label{def2:e^D_U}
For an open subset $U \subset X$
and under {\em Assumption A},
consider the bilinear form $\e^D_U$  
\[ \e^D_U(f,g) = \e(f,g),  \quad f,g \in \F^0(U), \]
where the domain $D(\e^D_U) = \F^0(U)$ is the closure of the space 
$\F_{\mbox{\em\tiny c}}(U)$ in the norm 
$(\e^{\mbox{\em \tiny s}}(f,f) + \Vert f \Vert_2)^{\frac{1}{2}}$.
\end{definition}
Under Assumption A, the form $(\e^D_U,D(\e^D_U))$ is closed, bounded below,
local and regular.

\begin{definition}
Let $X,\e^{\mbox{\em \tiny s}},\F,d,\mu$ be as in {\em Section \ref{sec-geo}}.
Let $(\e,\F)$ be a form satisfying {\em Assumption A} and $\e(f,f)\ge 0$ for all 
$f\in \F$.   
Let $U$ be a domain in $X$. 
A function $h\in L^2_{\mbox{\em\tiny loc}}(U)$  
is called an \emph{$\e$-harmonic profile} in $U$ if it satisfies the 
following properties:
\begin{enumerate}  
\item $h$ is a weak solution of $Lu=0$ in $U$;
\item $h\in \mathcal F^0_{\mbox{\em\tiny loc}}(U)$;
\item $h>0$ in $U$. 
\end{enumerate}
Fix a domain $W\subset U$. We say that a function $h$ in 
is a \emph{$(U,W)$-profile for $\e$} if $h$ is defined in $W$ and
\begin{enumerate}  
\item  $h$ is a weak solution of $Lu=0$ in $W$;
\item $h\in \mathcal F^0_{\mbox{\em\tiny loc}}(U,W)$;
\item $h>0$ in $W$. 
\end{enumerate}  
\end{definition}

\begin{proposition} 
Let $X,\e^{\mbox{\em \tiny s}},\F,d,\mu$ be as in {\em Section \ref{sec-geo}}.
Let $(\e,\F)$ be a form satisfying {\em Assumption A}  
and which is a Dirichlet form. Fix domains  $W\subset U$. 
Assume that the volume doubling property and Poincar\'e inequality  
hold locally on $\overline{U}$.   
\begin{enumerate} 
 \item  Assume that $U$   
is unbounded and  inner uniform near $W$. Then 
there exists a function $h$ which is a local weak  
solution of $Lh=0$ in $U$ and is a $(U,W)$-profile. 
\item If $U$ is unbounded and locally inner uniform it admits a 
harmonic profile $h$.  
\item If $U$ is bounded, inner uniform, $x_0\in U$, and 
$W \subset U\setminus B_U(x_0,\epsilon)$ then
the Green function $h(x)=G_{U}(x,x_0)$ is a $(U,W)$-profile.
\end{enumerate}
\end{proposition}
\begin{proof}  
Note that $G_U$
denotes the Green function in $U$ with Dirichlet boundary condition, i.e., 
the Green function for the form $(\e^D_U, D(\e^D_U))$. 
The third statement follows immediately from \cite[Lemma 3.9]{LierlSC1}.
See also \cite[Lemma 4.7]{GyryaSC} and \cite[Lemma 3.10]{LierlSC1}.
Note that applying the results of \cite{LierlSC1} requires assuming
that the form $\e$ is a (possibly non-symmetric) Dirichlet forms. We 
can extend these results to the general case using the results 
obtained at the end of the next section.

The idea of the proof of (i)-(ii) is to construct the profile $h$ has a 
limit of normalized Green functions. The details follow the 
same line of reasoning as in \cite[Theorem 4.16]{GyryaSC} with simple 
adaptations using Assumption A to take care of the fact that the form $\e$ 
is not symmetric. \end{proof}

\begin{proposition} \label{pro-hcontrol}
Let $X,\e^{\mbox{\em \tiny s}},\F,d,\mu$ be as in {\em Section \ref{sec-geo}}.
Let $(\e,\F)$ be a form satisfying {\em Assumption A}  
and which is a Dirichlet form. 
Fix $R>0$. Let $W\subset U$ be domains in $X$.  
Assume that the volume doubling property and 
the Poincar\'e inequality hold locally up to scale $R$ on $\overline{W}$ 
and that $U$ is locally inner uniform up to scale $R$ near $W$.
Let $h$ be a $(U,W)$-profile. Then there are constants
$K_0,K_1$ such that for any inner ball $B_U(x,r)$ with 
$0< K_0r <R$, $B_U(x,K_0 r)\subset W^{\sharp}$, we have
$$ \forall\,y\in B_U(x,r),\;\;  h(y)\le K_1 h(x_r) $$
where 
$x_r$ is any point  with $d_U(x,x_r)=r/4$ and $d(x_r, X\setminus W)\ge c_u r/8$.
The constants $K_0$ depend only on the local 
inner uniformity constants $c_u,C_u$. The constant $K_1$ depends only on $c_u,C_u$,the doubling and Poincar\'e constants
up to scale $R$ in $\overline{W}$,
the constants $C_0$--$C_5$ which give control over the skew-symmetric part 
and the killing part of the Dirichlet form $\e$ and  
an upper bound on $C_8(\e) R^2$.
\end{proposition}
\begin{proof} Compare the ratios for $h$ and  an 
appropriately chosen Green function. The result follows as an 
immediate corollary of Theorem \ref{thm2:bHP for u}
and the Green function estimates obtained in \cite[Lemmas 3.11-3.12]{LierlSC1}.
See also \cite[Theorem 4.17]{GyryaSC}.
\end{proof}

The next two propositions are straightforward applications 
of Proposition \ref{pro-hcontrol}. The proofs follow the same line of reasoning 
as in \cite[Theorem 4.17]{GyryaSC} and are omitted.

\begin{proposition}[Unbounded domains]
Let $X,\e^{\mbox{\em \tiny s}},\F,d,\mu$ be as in {\em Section \ref{sec-geo}}.
Let $(\e,\F)$ be a form satisfying {\em Assumption A} and which is a 
Dirichlet form.  Fix an unbounded domain $U$ in $X$ and let $h$ be a
$\e$-harmonic profile for $U$.
\begin{enumerate}
\item Assume $U$ is locally inner uniform and that
the volume doubling property and the Poincar\'e inequality 
hold locally in a  neighborhood $Y$ of $\overline{U}$ in $X$.
Then for any compact $K\subset \widetilde{U}$ there exist $r_K,\epsilon_K>0$ 
and  $C_K$ such that for any $x\in K$, $r\in (0,r_K)$, we have
$$ \forall\,y\in B_U(x,r),\;\;  h(y)\le C_K h(x_r) $$
where $x_r\in B_U(x,r)$ is any point 
with $d_U(x,x_r)=r/4$ and $d(x_r, \partial U)\ge \epsilon_K r$. 
Further
$$V_{h^2}(x,r)=\int_{B_U(x,r)}h^2d\mu\simeq h(x_r)^2V(x,r).$$

\item Fix $R>0$ and assume $U$ is locally $(c_u,C_u)$-inner uniform up to 
scale $R$
and that the volume doubling property and the Poincar\'e inequality 
hold up to scale $R$ in $\overline{U}$.
Then there exists 
$a_0,A_1\in (0,\infty)$ such that for any $x\in \widetilde{U}$, $r\in (0,a_0 R)$, we have
$$ \forall\,y\in B_U(x,r),\;\;  h(y)\le A_1 h(x_r) $$
where $x_r\in B_U(x,r)$ is any point 
with $d_U(x,x_r)=r/4$ and $d(x_r, \partial U)\ge  c_ur/8$. Further
  $$V_{h^2}(x,r)=\int_{B_U(x,r)} h^2d\mu\simeq h(x_r)^2V(x,r).$$
The constant $a_0$ depends only on $(c_u,C_u)$. The constant $A_1$
depends only on $c_u,C_u$, the volume doubling and Poincar\'e constant
up to scale $R$ in $\overline{U}$,
the constants $C_0$--$C_5$ which give control over the skew-symmetric part 
and the killing part of the Dirichlet form $\e$, and  
an upper bound on $C_8(\e) R^2$.
\end{enumerate}
\end{proposition}

\begin{proposition}[Bounded domains]
Let $X,\e^{\mbox{\em \tiny s}},\F,d,\mu$ be as in {\em Section \ref{sec-geo}}.
Let $(\e,\F)$ be a form satisfying {\em Assumption A} and which is a 
Dirichlet form.  Fix a bounded domain $U$ in $X$.
Fix $R\in (0, \frac{1}{2}\mbox{\em diam}_U)$ and assume $U$ is locally 
$(c_u,C_u)$-inner uniform at scale $R$
and that the volume doubling property and the Poincar\'e inequality 
hold up to scale $R$ in $\overline{U}$. 
Then there exist $A_0,A_1$ such that for any 
$\xi\in \widetilde{U}\setminus U$ and any 
$(U,B_U(\xi,R))$-profile  $h$,  we have, for all $x,r$ such that 
$B_U(x, A_0 r)\subset B_U(\xi,R)$,
$$  \forall\,y\in B_U(x,r),\;\;  h(y)\le A_1 h(x_r) $$
where $x_r\in B_U(x,r)$ is any point 
with $d_U(x,x_r)=r/4$ and $d(x_r, \partial U)\ge c_ur/8$.
Further
$$V_{h^2}(x,r)=\int_{B_U(x,r)}h^2d\mu\simeq h(x_r)^2V(x,r).$$
The constant $A_0$ depends only $c_u,C_u$. The constant $A_1$
depends only on $c_u,C_u$, the volume doubling and Poincar\'e constant
up to scale $R$ in $\overline{U}$,
the constants $C_0$--$C_5$ which give control over the skew-symmetric part 
and the killing part of the Dirichlet form $\e$, and  
an upper bound on $C_8(\e) R^2$.
\end{proposition}

\section{The $h$-transform technique} \label{sec:h-transform}
Throughout this section, we let 
$X,\e^{\mbox{\tiny s}},\F,d,\mu$ be as in Section \ref{sec-geo}
and fix a form $(\e,\F)$ satisfying  Assumption A. We 
also fix a domain $U$ in $X$. In addition, we fix a subdomain $W$ of $U$.
Recall that $W^{\sharp} $ is the largest open set in $\widetilde{U}$
whose intersection with $U$ is $W$. 

\begin{remark}
Any inner ball in $(\widetilde W,d_W)$ that lies in $W^{\sharp}$ is also an 
inner ball in $(\widetilde U,d_U)$. However, the metrics $d_W$ and $d_U$ 
do not necessarily coincide on $W^{\sharp}$.
\end{remark}

We assume the model form $(\e^{\mbox{\tiny{s}}},\F)$ satisfies (A1)-(A2-$\overline{W}$), and that the volume 
doubling property and the Poincar\'e inequality hold locally on  
$\overline{W}$ in $X$. 

\subsection{Some structural properties of $h$-transforms}

\begin{definition} \label{defEh}
Let $h$ be positive continuous on $W$.
Let $H: f\in L^2(W,h^2d\mu)\rightarrow L^2(W,d\mu)$,
$f\mapsto Hf= hf$.  Let $(\e^{D,W}_h,D(\e^{D,W}_h))$ be the form 
$$\e^{D,W}_h(f,g)=\e(Hf,Hg), \quad f,g\in H^{-1}(\F^0(W))=
D(\e^{D,W}_h).$$
\end{definition}

\begin{remark}\label{rem-defEh}
Dropping the reference to the Dirichlet condition and $W$ an 
writing $\e_h=\e^{D,W}_h$, observe that:
\begin{enumerate}
\item We have
$$\e_h^{\mbox{\tiny sym}}(f,g)
= \e^{\mbox{\tiny sym}}(hf,hg),\;\;
\e_h^{\mbox{\tiny skew}}(f,g)= \e^{\mbox{\tiny skew}}(hf,hg).$$
Further, for $f,g \in \mathcal F_{\mbox{\tiny{c}}}(W)\cap \mathcal C(W)$,
\begin{eqnarray*}
\e^{\mbox{\tiny s}}_h(f,g)&=&\e_h^{\mbox{\tiny sym}}(f,g) -\e_h^{\mbox{\tiny sym}}(fg,1)
\\
&=& \e^{\mbox{\tiny s}}(hf,hg)- \e^{\mbox{\tiny s}}(hfg,h) =\int h^2d\Gamma(f,g).
\end{eqnarray*}
Note that $\e^{\mbox{\tiny s}}_h$ must be understood as being the 
symmetric strongly local part of $\e_h$ which, in general,  
is not the same as the $h$ transform $(\e^{\mbox{\tiny s}})_h$ of 
$\e^{\mbox{\tiny s}}$. The two are the same exactly 
when $\e^{\mbox{\tiny s}}(hfg,h)=0$ for any $f,g\in \F_{\mbox{\tiny{c}}}(W)
\cap \mathcal C(W)$. 
This is the case when $h$ is a $(U,W)$-profile for $\e^{\mbox{\tiny s}}$.

\item Recall that $\l_h(f,g)=\frac{1}{2}(\e_h^{\mbox{\tiny skew}}(fg,1)+
\e_h^{\mbox{\tiny skew}}(f,g)).$ Since $\e$ is adapted,
$\l$ satisfies the  Leibniz rule  $\l(uf,v)=\l(u,fv)+\l(f,uv)$, 
$u,v,f\in \F_{\mbox{\tiny{c}}}(W)\cap \mathcal C(W)$.  Hence, we obtain
\begin{eqnarray*}
2\l_h(f,g)&=&  \l (hfg,h)+\r(hfg,h)+\l (hf,hg)+\r(hf,hg)\\
&=&  \l (hfg,h) -\l(h,hfg)+\l (hf,hg)-\l(hg,hf)\\
&=&  \l(f,gh^2) +\l (hg,hf)  -\l(h,hfg)\\&&+\l (h,hfg)+\l (f,h^2g)
-\l(hg,hf)\\
&=& 2\l(f,h^2g).
\end{eqnarray*}
This shows that $\l_h$ satisfies the appropriate Leibniz rule and chain rule
as in Definition \ref{def-firstorder}.

\item Under the additional assumption that $h$ is a $(U,W)$-profile for $\e$,
we have for $f,g\in \F_{\mbox{\tiny{c}}}(W) \cap \mathcal C (W)$,
$$\e_h^{\mbox{\tiny sym}}(fg,1)= \e^{\mbox{\tiny sym}}   (hfg,h)=
\e^{\mbox{\tiny sym}}   (h,hfg)=\e^{\mbox{\tiny skew}}   (hfg,h) $$
because $\e (h,hfg)=0$ under the present condition on $h,f,g$.
\item Assume that $h$ is an $\e^{\mbox{\tiny s}}$-$(U,W)$
profile.
Then, for all $f,g\in \F_{\mbox{\tiny{c}}}(W)\cap \mathcal C (W)$, we have
$\e^{\mbox{\tiny s}} (h,hfg)=0$. Hence, in this case,  
$$\e_h^{\mbox{\tiny sym}}(fg,1)=
\e^{\mbox{\tiny sym}}(h,hfg) = \e^{\mbox{\tiny sym}}(h^2fg,1).$$

\item  Assume that $h\in \F_{\mbox{\tiny loc}}$ is positive and continuous 
on $W$ and satisfies
$$\forall \,u\in \F_{\mbox{\tiny c}}(W),\;\;\e(f,u)= \gamma \int h u \, d\mu$$
for some $\gamma\in \mathbb R$.
Then, for all $f,g\in \F_{\mbox{\tiny{c}}}(W) \cap \mathcal C (W)$, we have  
$$\e_h^{\mbox{\tiny sym}}(fg,1)=
\e^{\mbox{\tiny sym}}(h,hfg) = \e^{\mbox{\tiny skew}}(hfg,h) +\gamma\int 
fg h^2d\mu.$$ 
\end{enumerate}
\end{remark}

\begin{definition} For a fixed $h$, positive and continuous on $W$,
let $$(\e^{D,W,h^2},D(\e^{D,W,h^2}))$$ be the Dirichlet form 
on $L^2(W,h^2d\mu)$
obtained by closing
$$\e^{D,W,h^2}(f,g)=\int h^2d\Gamma(f,g), \;\;f,g \in \F_{\mbox{\em\tiny c}}(W).$$
Let $\F^{h}=D(\e^{D,W,h^2})$ be the domain of this form, that is, 
the closure of $\F_{\mbox{\em\tiny c}}(W)$ for the norm
$$\|f\|_{\F^h}=\F^h_W
=\left(\int_W |f|^2 h^2d\mu +\int_W h^2d\Gamma(f,f)\right)^{1/2}.$$
\end{definition}
Note that, by definition,
$(\e^{D,W,h^2},\F^{h})$
is a symmetric strongly local regular Dirichlet form on $W$.

\begin{lemma} \label{prop:5.7 h unbd} Assume that $h$ is continuous 
positive on $W$.
Then the set 
$$H^{-1} \big(\F_{\mbox{\emph{\tiny{c}}}}(W) \cap L^{\infty}(W,\mu)\big)$$ 
is dense in the Hilbert space 
$$D(\e^{D,W}_{h}) = H^{-1} (\F^0(W)), \;\; \|f\|^2_{D(\e^{D,W}_h)}=
\int_W d\Gamma(hf,hf)+\int_Wh^2|f|^2d\mu $$ and
 \begin{eqnarray*}
 H^{-1}\big( \F_{\mbox{\emph{\tiny{c}}}}(W) \cap L^{\infty}(W,\mu) \big)  
&= & \F_{\mbox{\emph{\tiny{c}}}}(W) \cap L^{\infty}(W,\mu) \\
&=&  \F_{\mbox{\emph{\tiny{c}}}}(W) \cap L^{\infty}(W,h^2 \mu). \end{eqnarray*}
In particular, $\F_{\mbox{\emph{\tiny{c}}}}(W) \cap L^{\infty}(W,\mu) \big) $ is also dense in 
$D(\e^{D,W}_{h}) = H^{-1} (\F^0(W))$.
\end{lemma}
\begin{proof}
We follow \cite[Proposition 5.7]{GyryaSC}. The set $\F_{\mbox{\tiny{c}}}(W) \cap L^{\infty}(W,\mu)$ is dense in the Hilbert space $\F^0(W)$. Since $H$ is a unitary operator between the Hilbert spaces $D(\e_{h})$ and $\F^0(W)$, it follows that $H^{-1}\big(\F_{\mbox{\tiny{c}}}(W) \cap L^{\infty}(W,\mu)\big)$ is also dense in the Hilbert space $D(\e_{h})$. Since $h$, $1/h$ are both in 
$\F_{\mbox{\tiny{loc}}}(W) \cap L^{\infty}_{\mbox{\tiny{loc}}}(W,\mu)$, 
the equality
 \[ H^{-1}\big( \F_{\mbox{\tiny{c}}}(W) \cap L^{\infty}(W,\mu) \big) 
    = \F_{\mbox{\tiny{c}}}(W) \cap L^{\infty}(W,\mu) \]
follows from the fact that $\F_{\mbox{\tiny{loc}}}(W) \cap L^{\infty}_{\mbox{\tiny{loc}}}(W,\mu)$ is an algebra. 
\end{proof}

\begin{lemma} \label{lem-compEh}
Assume  that $h$ is a $(U,W)$-profile for either  
$\e+\gamma \langle\cdot,\cdot\rangle$ 
or  $\e^{\mbox{\em\tiny s}}+\gamma\langle\cdot,\cdot\rangle$, 
$\gamma\in \mathbb R$.
Then there exists $C\in (0,\infty)$ such that, 
for $f,g\in \F_{\mbox{\em \tiny{c}}}(W) \cap L^\infty(W,\mu)$, we have
\begin{enumerate}
\item  
$\int_Wd\Gamma(hf,hf)\le C\left(\e^{D,W,h^2}(f,f)+\int_W |f|^2h^2d\mu\right).$
\item  
$|\e^{D,W}_h(f,f)|\le C\left(\e^{D,W,h^2}(f,f)+\int_W |f|^2h^2d\mu\right).$
\item 
$\e^{D,W,h^2}(f,f)\le C\left(\e^{D,W}_h(f,f)+ \int_W |f|^2 h^2d\mu\right).$
\end{enumerate}
In particular, $D(\e^{D,W}_h)=D(\e^{D,W,h^2})=\F^{h}$.
\end{lemma}
Note that  $\F_{\mbox{\tiny{c}}}(W) \cap L^\infty(W,\mu)$ is dense in the domains of both forms 
$\e^{D,W}_h$ and $\e^{D,W,h^2}$ so that the last statement follows from 
(ii)-(iii).
\begin{proof}  
We give the proof for (ii) and (iii) when $h$ is a $(U,W)$-profile for $\e$. 
The proof of (i) and  the cases when $h$ is a $(U,W)$-profile for $\e+\gamma\langle\cdot,\cdot\rangle$, $\gamma\neq 0$  or  $\e^{\mbox{\tiny s}}+\gamma\langle\cdot,\cdot\rangle$ are similar. To simplify notation, we set
$\e^{D,W}_h=\e_h$ and $\e^{D,W,h^2}=\e^{h^2}$ and we drop the explicit 
reference to the Dirichlet condition and the set $W$.
For any $f \in \F_{\mbox{\tiny{c}}}(W) \cap L^{\infty}(W,\mu)$, 
Assumption A(i) yields 
\begin{align*}
|\e_h(f,f) | &=|\e^{\mbox{\tiny sym}}(hf,hf)|
 \le  \int d\Gamma(hf,hf)  + |\e^{\mbox{\tiny{sym}}}(h^2 f^2,1)| \\
& \leq C \left( \int h^2 d\Gamma(f,f) + \int f^2 d\Gamma(h,h) + \int h^2 f^2 d\mu \right).
\end{align*}
The constant $C$ depends only on $C_2$, $C_3$.
Because $\e(h,hf^2)=0$, we get from Lemma \ref{caccio} 
that for any $k_1,k_2,k_3 > 0$,
\begin{align*}
 \left( 1 - \frac{1}{k_1} - 2k_2 C_2 \right) \int f^2 d\Gamma(h,h)
\leq & \left(4k_1 + 2k_2 C_2 + k_3 \right) \int h^2 d\Gamma(f,f)  \\
& + \left( \frac{1}{k_2} + k_2C_3 + \frac{C_5}{k_3} \right) \int f^2 h^2 d\mu.
\end{align*}
Hence (with a different $C$ depending only on $C_2,C_3, C_5$),
\begin{align*}
|\e_h(f,f)| 
& \leq C \left( \e^{h^2}(f,f) + \int h^2 f^2 d\mu \right).
\end{align*}
This proves (ii).

To prove (iii), we use the fact that,
for $f\in \F_{\mbox{\tiny{c}}}(W) \cap L^{\infty}(W,\mu)$,
 $\e(h,hf^2)=0$, and  Assumption A(ii) to obtain
\begin{align*}
\e^{h^2}(f,f) 
& = \e_{h}(f,f) - \e^{\mbox{\tiny{sym}}}(h^2 f^2,1) - \e^{\mbox{\tiny{s}}}(h,hf^2) \\
& = \e_{h}(f,f) + \e^{\mbox{\tiny{skew}}}(h,hf^2) \\
& \leq \e_{h}(f,f) + k_4 \int h^2 d\Gamma(f,f) + \frac{C_5}{k_4} \int f^2 h^2 d\mu,
\end{align*}
where $k_4 > 0$ is arbitrary. Choosing $k_4=1/2$, we get
\begin{align*}
\e^{h^2}(f,f)
& \leq 2 \left( \e_{h}(f,f) + 2C_5\int f^2 h^2 d\mu \right).
\end{align*}
\end{proof}

\begin{proposition} \label{prop:5.8 h} Assume that $h$ is continuous 
positive on $W$ and belongs to $\F^0(U,W)$.
The strongly local Dirichlet form $\big( \e^{D,W,h^2}, D\big(\e^{D,W,h^2}\big) \big)$ is regular on $(W^{\sharp}, h^2 d\mu)$ with core $\mbox{\emph{Lip}}_{\mbox{\em{\tiny{c}}}}(W^{\sharp},d_W)$.
\end{proposition}

\begin{proof}
We follow the proof of \cite[Proposition 5.8]{GyryaSC}.
As $\big( \e^{D,W,h^2}, D\big(\e^{D,W,h^2}\big) \big)$ is regular on $W$, 
$C_{\mbox{\tiny{c}}}( W ) \cap D\big(\e^{D,W,h^2}\big)$ is dense in $D\big(\e^{D,W,h^2}\big)$. So we only need to show that $C_{\mbox{\tiny{c}}}( W^{\sharp} ) \cap D\big(\e^{D,W,h^2}\big)$ is dense in $C_{\mbox{\tiny{c}}}(W^{\sharp})$ in the $\sup$ norm. Consider a function $f \in \textrm{Lip}_{\mbox{\tiny{c}}}(W^{\sharp},d_W)$ with Lipschitz constant $k$. As $\textrm{Lip}_{\mbox{\tiny{c}}}(W^{\sharp},d_W)$ is dense in $C_{\mbox{\tiny{c}}}(W^{\sharp})$ in $\sup$ norm, it suffices to show that $f \in D\big(\e^{D,W,h^2} \big)$. In view of Lemma 
\ref{prop:5.7 h unbd}, 
it suffices to show that $f h \in \F^0(W)$.
Since $\textrm{Lip}_{\mbox{\tiny{c}}}(W^{\sharp},d_W) \subset \textrm{Lip}(W,d_W)$, Corollary \ref{cor:2.22} implies that $f \in \F_{\mbox{\tiny{loc}}}(W)$ and
 \[ \Upsilon(f,f) = \frac{d\Gamma(f,f)}{d\mu} \leq k^2  \quad \mbox{ almost everywhere on } W. \] 
Since $f$ is bounded, this shows that $f \in \F(W)$. 
Let $V \subset W$ be an open set containing $\textrm{supp}(f) \cap W$ 
and relatively compact in $W^{\sharp}$ with the property that 
$\textrm{supp}(f) \subset V^{\sharp} \subset W^{\sharp}$. Applying Lemma \ref{lem:2.46} with $g=h \in \F^0(W,V)$, we obtain that $f h \in \F^0(W)$.
\end{proof}

\begin{definition} Assume that $h$ is continuous 
positive on $W$ and belongs to $\F^0(U,W)$. Recall that
$$ \F^{h}  =  \F^h_W= D(\e^{D,W,h^2}).$$
For an open subset $V \subset W^{\sharp}$, let
\begin{align*}
 \F^h_{\mbox{\em\tiny{{loc}}}}(V)  =  
\{ & f \in L^2_{\mbox{\em\tiny{loc}}}(V,h^2 d\mu) : \forall \textrm{ compact } K \subset V, \ \textrm{ there exists } f^{\sharp} \in \F^h \\
& \textrm{ so that } f = f^{\sharp}\big|_K \textrm{ a.e.} \}.
\end{align*}
Similarly, define $\F^h(V)$ and $\F^h_{\mbox{\em \tiny{c}}}(V)$ 
in terms of $\big(\e^{D,W,h^2},D\big(\e^{D,W,h^2}\big)\big)$.
\end{definition}

\begin{remark} \label{rem:e_h^s}
\begin{enumerate}
\item  By Proposition \ref{prop:5.8 h} and Lemmas \ref{prop:5.7 h unbd},
\ref{lem-compEh},
we have
$$\F^h_W= H^{-1} (\F^0(W))=D(\e_h).$$ 
\item It is now plain that the symmetric strongly local regular Dirichlet form
$\big( \e^{D,W,h^2},\F^h_W \big)$ is the strongly local part of the symmetric part of the form $(\e_h,D(\e_h))$. In particular, for any $f \in \F^h_{\mbox{\tiny{c}}}(W^{\sharp})\cap \mathcal C(W^{\sharp})$,
\begin{align*}
\e_h^{\mbox{\tiny{s}}}(f,f) = \e_h^{\mbox{\tiny{sym}}}(f,f) - \e^{\mbox{\tiny{sym}}}_h(f^2,1) = \int h^2 d\Gamma(f,f) = \e^{D,W,h^2}(f,f).
\end{align*}
\item The space $\textrm{Lip}_{\mbox{\tiny{c}}}(W^{\sharp},d_U)$ is contained in $\textrm{Lip}_{\mbox{\tiny{c}}}(W^{\sharp},d_W)$, because for any $x,y \in W^{\sharp}$ it holds $d_U(x,y) \leq d_W(x,y)$. In fact, both spaces are the same. To see this, observe that for any $f \in \textrm{Lip}_{\mbox{\tiny{c}}}(W^{\sharp},d_W)$ with Lipschitz constant $C_W$ and any $x,y \in W^{\sharp}$ with $d_U(x,y)$ strictly less than $d_W(x,y)$ we have
 \[ |f(x) - f(y)| \leq C \big( d_U (x,\partial W \cap U) + d_U(y,\partial W \cap U) \big) \leq C d_U(x,y), \]
where
\[ C = \frac{\max_{z \in W^{\sharp}} f(z)}{d_U(\textrm{supp}(f),\partial W \cap U)}. \]
Hence, $f$ is in $\textrm{Lip}_{\mbox{\tiny{c}}}(W^{\sharp},d_U)$ with Lipschitz constant $C_U = \max\{C_W, C \}$.
\end{enumerate}
\end{remark}

\begin{lemma} \label{lem:A1 for d_U}
Assume that $h$ is continuous 
positive on $W$ and belongs to $\F^0(U,W)$. The metrics $d_U$, $d_W$ and $d_{\e^{D,W,h^2}}$ coincide on any inner ball $B=B_{\widetilde U}(a,r)$ with $B_{\widetilde U}(a,3r) \subset W^{\sharp}$. 
\end{lemma}

\begin{proof}
Clearly, the inner metrics $d_W$ and $d_U$ coincide on the ball $B$, since $B$ is far away from $U \setminus W$.
We follow the line of reasoning in \cite[Proof of Lemma 3.32]{GyryaSC} to show that the inner metrics coincide with $d_{\e^{D,W,h^2}}$.
Fix $y,z \in B$. Then the cut-off function
 \[ \rho_y(x) = \max \{ d_W(y,z) - d_W(y,x), 0 \}, \]
is a compactly supported Lipschitz function on $(W^{\sharp},d_W)$ and 
$\rho_y \in \F^{h}_{\mbox{\tiny{loc}}}(W^{\sharp}) \cap C(W^{\sharp})$ by Proposition \ref{prop:5.8 h}. Moreover, $\Upsilon(\rho_y,\rho_y) \leq 1 < \infty$ a.e.~on $W^{\sharp}$ by Corollary \ref{cor:2.22}.
Thus, 
\[ d_W(y,z) = \rho_y(y) - \rho_y(z) \leq d_{\e^{D,W,h^2}}(y,z). \]

We now show the opposite inequality. Any two points $y,z \in B \cap W$ can be connected by a curve $\gamma=\gamma_{y,z}$ in $W$ without self-intersections. Let $A_{\gamma}$ be an open, relatively compact subset of $W$ that contains the curve. By \cite[Theorem 3]{Stu95geometry} (recall that (A1) holds on $(X,\mu,\e^{\mbox{\tiny{s}}},D(\e))$), we have
\begin{align*}
\textrm{length}(\gamma) 
& = \sup  \{ u(y) - u(z) : u \in \F_{\mbox{\tiny{loc}}}(A_{\gamma}) \cap C(A_{\gamma}), d\Gamma(u,u) \leq d\mu \} \\
& = \sup \{ u(y) - u(z) : u \in \F^h_{\mbox{\tiny{loc}}}(A_{\gamma}) \cap C(A_{\gamma}),  h^2 d\Gamma(u,u) \leq h^2 d\mu \} \\
& \geq d_{\e^{D,W,h^2}}(y,z).
\end{align*}
Hence, $d_W(y,z) = \inf_{\gamma} \mbox{length}(\gamma) \geq d_{\e^{D,W,h^2}}(y,z)$ for all $y,z \in B \cap W$. To show that $d_W$ and $d_{\e^{D,W,h^2}}$ coincide on $B$, approximate $y$ and $z$ by points in $B \cap W$. 
\end{proof}

\begin{lemma} \label{lem:Assumption 1 for e_h} 
Assume 
that $h$ is a $(U,W)$-profile 
for either 
$\e+\gamma\langle\cdot,\cdot\rangle
$ or $\e^{\mbox{\em\tiny s}}+\gamma\langle\cdot,\cdot\rangle$.
Then the form $\big(\e_h, D(\e_h) \big)$ satisfies {\em Assumption A}
on $(W^{\sharp}, h^2 d\mu)$ with respect to  $(\e^{D,W,h^2},\F^h_W)$.
Further: 
\begin{enumerate}
\item If $h$ is a 
$(U,W)$-profile for $\e^{\mbox{\em \tiny s}}+\gamma\langle \cdot,\cdot\rangle$, 
then the sector condition constant $C_0(\e_h)$ and the
constants $C_2(\e_h),C_3(\e_h),C_5(\e_h)$ for the form $\e_h$
on $(W^{\sharp}, h^2 d\mu)$ with respect to  $(\e^{D,W,h^2},\F^h_W)$
are all bounded as follows:
$$C_0(\e_h)\le C_0(\e)(1+|\gamma|),\;C_2(\e_h)\le C_2(\e),$$
$$C_3(\e_h)\le C_2(\e)|\gamma|+C_3(\e)+|\gamma|^2, C_5(\e_h)\le C_5(\e)$$
and $C_8(\e_h)\le 4(C_8(\e)+|\gamma|)$. 
\item  If $h$ is a $(U,W)$-profile for $\e+\gamma\langle\cdot,\cdot\rangle$, then
the sector condition constant $C_0(\e_h)$ and the
constants $C_2(\e_h),C_3(\e_h),C_5(\e_h)$ for the form $\e_h$
on $(W^{\sharp}, h^2 d\mu)$ with respect to  $(\e^{D,W,h^2},\F^h_W)$
are all bounded in terms of an upper bound for $C_0(\e),C_2(\e),C_3(\e),C_5(\e)$ 
and $|\gamma|$.
\end{enumerate}
\end{lemma}
\begin{proof}
We only treat the case when $h$ is a $(U,W)$-profile for $\e$ 
(i.e., $\gamma=0$). The other cases
are similar. We refer the reader to Remark \ref{rem-defEh} for 
various algebraic computations regarding $\e_h$ that are relevant to this proof.
Let $f,g\in \F^h_W\cap \mathcal C_{\mbox{\tiny c}}(W)$.
We have 
$$\e_h(f,g)
=\e^{\mbox{\tiny s}}_h(f,g)+\e^{\mbox{\tiny skew}}_h(f,g)+ 
\e^{\mbox{\tiny sym}}_h(fg,1)
$$
with
$$\e^{\mbox{\tiny s}}_h
=\e^{D,W,h^2},\;\; \e^{\mbox{\tiny skew}}_h(f,g)=
\e^{\mbox{\tiny skew}}(hf,hg)$$
and
$$ 
\e^{\mbox{\tiny sym}}_h(fg,1) = \e^{\mbox{\tiny sym}}(hfg,h) = 
\e^{\mbox{\tiny skew}}(hfg,h)
$$
where the last equality follows from the fact that $\e(h,hfg)=0$
and needs to be modified appropriately when $h$ is a profile for 
a form different from $\e$. See Remark \ref{rem-defEh}(iii)-(iv).

Using the isometry $H:L^2(W,h^2\mu)\rightarrow L^2(W,\mu)$ and 
Lemma \ref{lem-compEh}
in an obvious way,
we see that  
$$|\e^{\mbox{\tiny skew}}(hf,hg)|\le C\|f\|_{\F^h}\|g\|_{\F^h}.$$ 
Next, we check that
$$|\e^{\mbox{\tiny skew}}(hfg,h)|\le C \|f\|_{\F^h}\|g\|_{\F^h}.$$ 
By Assumption A(ii), we can find a constant $k$ such that
the symmetric bilinear form 
$$(f,g)\mapsto 
\e^{\mbox{\tiny skew}}(hfg,h) +k \left(\int_W fg h^2 d\mu+ \int_W h^2d\Gamma(f,g)\right) $$
is positive definite. Applying the Cauchy-Schwarz inequality  
and Assumption A(ii) yield
$|\e^{\mbox{\tiny skew}}(hfg,h)|\le C \|f\|_{\F^h}\|g\|_{\F^h}$ as desired.
These computations also yield 
$|\e_h(fg,1)|+|\e_h(1,fg)|\le C \|f\|_{\F^h}\|g\|_{\F^h}$.
Together with \ref{rem-defEh}(ii), these estimates show 
that $\e_h$ is adapted to $(\e^{\mbox{\tiny s}}_h, \F^h$).

Next, we prove that $\e_h$ satisfies Assumption A(i)--(ii).
Assumption A(ii) for $\e$ yields 
\begin{align*}
 \big| \e_h^{\mbox{\tiny{sym}}}(f^2,1) \big| 
& 
  = \big| \e^{\mbox{\tiny{skew}}} (h,hf^2) \big| \\
& \leq 2 \left( \int h^2 d\Gamma(f,f) \right)^{\frac{1}{2}} 
\left(C_5 \int f^2 h^2 d\mu \right)^{\frac{1}{2}}.
\end{align*}
This is Assumption A(i) for $\e_h$.
Assumption A(ii) for $\e_h$ follows immediately from Assumption A(ii) for $\e$. 
The statements about the constants are simple bookkeeping.
See Remarks \ref{rem-defEh}(ii) and \ref{rem:e_h^s}(ii).
\end{proof}

\subsection{Properties of  $h$-transforms
in inner uniform domains}
We continue to work under the hypotheses made at the beginning of Section 
\ref{sec:h-transform}.

\begin{theorem} \label{h-VDP0}
Assume that
\begin{itemize}
\item The volume doubling property and the Poincar\'e inequality
(for the model form $(\e^{\mbox{\em\tiny s}},\F)$) hold locally up to scale $R$
on $\overline{W}$.
\item The domain $U$ is locally $(c_u,C_u)$-inner uniform up 
to scale $R$ near $W$.
\end{itemize}
Assume 
that $h$ is a $(U,W)$-profile for either $\e+\gamma\langle\cdot,\cdot\rangle$ 
or $\e^{\mbox{\em \tiny s}} +\gamma\langle\cdot,\cdot\rangle$. 
Then the following properties hold:
\begin{enumerate}
\item The symmetric strongly local regular Dirichlet form 
$\big( \e^{D,W,h^2},\F^h_W\big)$ satisfies property {\em (A1)}
and {\em (A2-$B$)} for any inner uniform ball $B=B_{\widetilde{U}}(a,r)$ such 
that $B_{\widetilde{U}}(a,3r)\subset W^{\sharp}$. 
\item There exist  constants $a_0,A_0\in (0,\infty)$ such that, 
for any $r\in (0, a_0R)$
and any inner ball $B=B_{\widetilde{U}}(a,r)$ with 
$B=B_{\widetilde{U}}(a,A_0r)\subset W^{\sharp}$, we have
$$ V_{h^2}(a,2r)\le D(W,R) V_{h^2}(a,r)$$
and, for any $f\in \F^h(B)$,
$$\int_{B}|f-f_B|^2h^2d\mu\le P(W,R) r^2\int_Bh^2d\Gamma(f,f).$$
\end{enumerate} 
The constants $a_0,A_0$ depend only $c_u,C_u$. If $h$ is a 
$(U,W)$-profile for $\e^{\mbox{\em \tiny s}} +\gamma\langle\cdot,\cdot\rangle$ 
then the constants 
$D(W,R)$ and $P(W,R)$ depend only on $c_u,C_u$, the volume doubling and 
Poincar\'e constants up to scale $R$ in $\overline{W}$,  
and an upper bound on $|\gamma|R^2$.
If $h$ is a 
$(U,W)$-profile for $\e +\gamma\langle\cdot,\cdot\rangle$ 
then the constants 
$D(W,R)$ and $P(W,R)$ depend only on $c_u,C_u$, the volume doubling and 
Poincar\'e constants up to scale $R$ in $\overline{W}$,  
and an upper bound on $C_0(\e),C_2(\e),C_3(\e),C_5(\e),|\gamma|$ and $R$.
\end{theorem}
The first assertion is clear by Lemma \ref{lem:A1 for d_U}.
The proof of the second assertion is done in two stages. The first stage 
concerns the case when $h$ is a profile relative to the Dirichlet form 
$\e^{\mbox{\tiny s}}$.
\begin{proof}[Proof in the case of a 
$\e^{\mbox{\emph{\tiny s}}}$-$(U,W)$-profile]
When $h$ is a $\e^{\mbox{\tiny s}}$-$(U,W)$-profile, we can apply 
Proposition \ref{pro-hcontrol} to obtain the asserted doubling property 
of the volume function $V_{h^2}$. Note that this very crucial step is based on 
the boundary Harnack principle for $\e^{\mbox{\tiny s}}$ 
(which has only been proved so far 
for Dirichlet forms).  Since the volume function $V_{h^2}$ has the doubling 
property, the stated Poincar\'e inequality follows by the line of reasoning 
explained in \cite[Theorem 3.13]{GyryaSC}. 
See also  \cite[Theorem 3.27]{GyryaSC}. One may have to change the 
constants $a_0,A_0$ when passing from the volume doubling property 
to the Poincar\'e inequality.  In this case, the constants $D(W,R)$ and $ P(W,R)$
depends only on the doubling and Poincar\'e constants for 
$(\e^{\mbox{\tiny s}},\F)$ 
up to scale $R$ in $\overline{W}$ and 
the inner uniformity constants $c_u,C_u$ up to scale $R$ near $W$.
\end{proof}

\begin{proof}[Proof of the case of a $(U,W)$-profile for 
$\e^{\mbox{\emph{\tiny s}}}+\gamma\langle\cdot,\cdot\rangle$]
Let $h$ be as in the first part of the proof, that is, 
a $\e^{\mbox{\tiny s}}$-$(U,W)$-profile.
Using the result proved in stage 1 together with 
Lemma \ref{lem:Assumption 1 for e_h} and 
\cite[Theorem 2.13]{LierlSC2}, it follows that there exist
$a_0,A_0$ such that the parabolic Harnack inequality  holds for 
the form $\e^{\mbox{\tiny s}}_h+\gamma\langle\cdot h,\cdot h\rangle$
on any inner ball 
$B=B_{\widetilde{U}}(a,r)$, $0<r<a_0 R$ with 
$B_{\widetilde{U}}(a,A_0 r)\subset W^{\sharp}$. Further, the 
Harnack constant  
depends only on $c_u,C_u$, the volume doubling and Poincar\'e constant 
for  $(\e^{\mbox{\tiny s}},\F)$  
up to scale $R$ on $\overline{W}$, and an upper bound on 
$|\gamma|R^2$.

In particular if $\hat{h}$ is a $(U,W)$-profile for 
$\e^{\mbox{\tiny s}}+\gamma\langle\cdot,\cdot\rangle$ then $\hat{h}/h$
is a positive harmonic function (in the weak sense in $W$) for 
$\e^{\mbox{\tiny s}}_h+\gamma\langle\cdot h,\cdot h\rangle$ and we have
$$\forall\,x,y\in B,\;\; c_H \frac{\hat{h}(y)}{h(y)}\le \frac{\hat{h}(x)}{h(x)}\le C_H\frac{\hat{h}(y)}{h(y)},$$ where 
$B$ is as above and $B_{\widetilde{U}}(a,A_0 r)\subset W^{\sharp}$. The constants  $c_H,C_H$ depend only on $c_u,C_u$, the volume doubling and Poincar\'e constant 
for  $(\e^{\mbox{\tiny s}},\F)$  
up to scale $R$ on $\overline{W}$, and an upper bound on 
$|\gamma|R^2$.
From this, it is clear that Theorem \ref{h-VDP0} also holds 
in the case of a $(U,W)$-profile for
$\e^{\mbox{\tiny s}}+\gamma\langle\cdot,\cdot\rangle$.
\end{proof}

\begin{proof}[Proof in the case of a $(U,W)$-profile for 
$\e+\gamma\langle\cdot,\cdot\rangle$]
The proof is the same as in the case of 
$(U,W)$-profile for
$\e^{\mbox{\tiny s}}+\gamma\langle\cdot,\cdot\rangle$.
However, in this case, The constant in the Harnack inequality 
for the form $\e_h+\gamma\langle\cdot h,\cdot h\rangle$.
depends on $c_u,C_u$, the volume doubling and Poincar\'e constant 
up to scale $R$ on $\overline{W}$ and 
an upper bound on $C_0(\e),C_2(\e),C_3(\e),C_5(\e),|\gamma|$ and $R$. 

In fact, this argument proves that the forms $\e^{\mbox{\tiny s}}+
\gamma \langle \cdot,\cdot\rangle$,  $\e$ and 
$\e+\gamma \langle \cdot,\cdot\rangle$ all  
satisfy the geometric  
boundary Harnack principle up to scale $R$ near $W$. 
It follows that Theorem \ref{thm2:bHP for u} still holds true without 
the hypothesis
(i) of Theorem \ref{thm2:bHP for u}.  Consequently, (i) and (ii)
of Theorem \ref{th-BHP1} hold true without assuming that 
$(\mathcal E, \F)$ is a Dirichlet form. See Theorem \ref{th-BHP2} 
below. \end{proof}

\if
Let $h$ still be as in the first part of the proof, that is, 
$h$ is a $\e^{\mbox{\tiny s}}$-$(U,W)$-profile.
Using the result proved in stage 1 together with 
Lemma \ref{lem:Assumption 1 for e_h} and 
\cite[Theorem 2.13]{LierlSC2}, it follows that there exist
$a_0,A_0$ such that the parabolic Harnack inequality  holds for 
$\e_h+\gamma \langle\cdot h,\cdot h\rangle$ 
on any inner ball 
$B=B_{\widetilde{U}}(a,r)$, $0<r<a_0 R$ with 
$B_{\widetilde{U}}(a,A_0 r)\subset W^{\sharp}$ (again, we may have 
to adjust the constants $a_0,A_0$). The constant in the Harnack inequality 
depends only on $c_u,C_u$, the volume doubling and Poincar\'e constant 
up to scale $R$ on $\overline{W}$, the constant 
$C_0(\e)$ and an upper bound on $(C_8(\e)+|\gamma|)(1+R^2)$.

In particular, if $\hat{h}$ denotes a fixed $(U,W)$-profile for 
$\e+\gamma\langle\cdot,\cdot\rangle$,
the Harnack inequality for 
$\e+\gamma\langle\cdot,\cdot\rangle$
applies to the weak solution $\hat{h}/h$.
Adjusting $a_0,A_0$ again if necessary, 
on any inner ball $B=B_{\widetilde{U}}(a,r)$ with 
$B=B_{\widetilde{U}}(a,A_0r)\subset W^{\sharp}$, we have
$$\forall\, x\in B_{\widetilde{U}}(a,r),\;\;\hat{h}(x)\le A_1 \hat{h}(a_r)$$
with $x_r$ any point in $B(a,r/4)$ with $d(x_r,\partial U)> c_u r/8$.  

In any case, the boundary control obtained above for $\hat{h}$ implies that
the volume function $V_{\hat{h}^2}$ is doubling in any inner balls $B(a,r)$
as above. By the same line of reasoning outlined in the first part of the proof, the Poincar\'e inequality holds with respect to the measure 
$\hat{h}^2d\mu$ as well.

\fi

As a corollary that has already been used in the proof above, 
we have the following very useful result.

\begin{theorem} \label{h-VDP}
Assume that: 
\begin{itemize}
\item The volume doubling property and the Poincar\'e inequality
(for the model form $(\e^{\mbox{\em\tiny s}},\F)$) hold locally up to scale $R$
on $\overline{W}$.
\item The domain $U$ is locally $(c_u,C_u)$-inner uniform up 
to scale $R$ near $W$.
\end{itemize}
\begin{enumerate}
\item
Assume that $h$ is a $(U,W)$-profile for  
$\e^{\mbox{\em\tiny s}}+\gamma\langle \cdot,\cdot\rangle$. 
Then there exist constant $a_0,A_0 $ such that
for any inner ball $B=B_{\widetilde{U}}(a,r)$ with $r\in (0,a_0R)$
and $B_{\widetilde{U}}(a,A_0r)\subset W^{\sharp}$, 
the parabolic Harnack inequality  for $\e_h$  holds in $B$ up to scale $r$, with
a parabolic Harnack constant which 
depends only on $c_u,C_u$, the volume doubling and Poincar\'e constant 
up to scale $R$ on $\overline{W}$, $C_0(\e)$, and an upper bound on 
$(C_8(\e)+|\gamma|)R^2$.
\item Assume that $h$ is a $(U,W)$-profile for  
$\e+\gamma\langle \cdot,\cdot\rangle$. 
Then there exist constants $a_0,A_0 $ such that
for any inner ball $B=B_{\widetilde{U}}(a,r)$ with $r\in (0,a_0R)$
and $B_{\widetilde{U}}(a,A_0r)\subset W^{\sharp}$, 
the parabolic Harnack inequality  for $\e_h$  holds in $B$ up to scale $r$, with
a parabolic Harnack constant which 
depends only on $c_u,C_u$, the volume doubling and Poincar\'e constant 
up to scale $R$ on $\overline{W}$, $C_0(\e)$, and 
an upper bound on $C_0(\e),C_2(\e),C_3(\e),C_5(\e),|\gamma|$ and $R$.
\end{enumerate}
\end{theorem}

Another useful result already mentioned and used in stages 2 and 3 
of the proof of
Theorem \ref{h-VDP} is the following extension of Theorem \ref{th-BHP1}
to the case when $\e$ is not a Dirichlet form.
\begin{theorem} \label{th-BHP2}
Let $X,\e^{\mbox{\em \tiny s}},\F,d,\mu$ be as in {\em Section \ref{sec-geo}}.
Let $(\e,\F)$ be a form satisfying {\em Assumption A}.    
Let $U$ be a domain in $X$.    
\begin{enumerate}
\item Fix a domain $W\subset U$, 
and assume that $U$ is locally inner uniform near $W$. Assume 
also that the volume doubling property and the Poincar\'e inequality hold locally in $\overline{W}$. 
Then the geometric  boundary Harnack principle holds locally in $U$ near $W$.
\item  Fix $R>0$ and a domain $W\subset U$.  
Assume that $U$ is locally $(c_u,C_u)$-inner uniform up to scale $R$ near $W$
and that the volume doubling
property and Poincar\'e inequality hold up to scale $R$ in $\overline{W}$. 
Then there exists $a_0 >0$ depending only on $c_u,C_u$ such that 
the geometric boundary Harnack principle holds locally up to scale 
$a_0R$ near $W$ with constants depending only on $c_u,C_u$, 
the volume doubling and Poincar\'e inequality constants 
up to scale $R$ in $\overline{W}$, $C_0(\e)$, 
and an upper bound on $C_8(\e)R^2$.
\end{enumerate}
\end{theorem}

\section{Estimates for the Dirichlet heat kernel} \label{sec:phi}

Let $(\e^{\mbox{\tiny{s}}},\F)$ be a model form as in Section \ref{sec-geo} and assume that it satisfies (A1)-(A2).
Let $\e$ be a form satisfying Assumption A.
Fix a domain $U$ and consider the bilinear form $(\e^D_U,\F^0(U))$
of Definition \ref{def2:e^D_U}.
In this section, we derive the main results of this paper which are
two-sided estimates  for the  kernel of the semigroup  
$P^{D}_{U,t}$ associated with $(\e^D_U,\F^0(U))$, that is, the heat kernel 
for $\e$ with Dirichlet boundary condition along $\partial U$.
For simplicity, let us assume that the volume doubling condition 
and the Poincar\'e inequality hold locally in $U$. This immediately implies 
that the semigroup $P^{D}_{U,t}$ admits a continuous positive kernel 
$p^{D}_U(t,x,y)$ in $U$, so that
$$P^{D}_{U,t}f(x)=\int p^{D}_U(t,x,y)f(y)dy.$$
In fact, by virtue of the local Harnack inequality, 
this kernel is locally H\"older continuous in 
$(t,x,y)\in (0,\infty)\times U\times U$.

Next, let $h$ be a positive continuous function in $U$ and consider the form
$\big(\e^{D,U}_h, H^{-1}(\F^0(U)\big)$, where the map $H:f\mapsto hf$ is the natural 
isometry between $L^2(U,h^2d\mu)$ and $L^2(U,d\mu)$ as in Definition 
\ref{defEh}. By construction, the form $\e_h^{D,U}$ induces a semigroup 
$P^{D,U}_{h,t}: L^2(U,h^2d\mu)\rightarrow L^2(U,h^2d\mu)$ given by
$$P^{D,U}_{h,t}f= h^{-1} P^D_{U,t}(hf).$$
Hence, this semigroup admits a kernel $p^{D,U}_h(t,x,y)$, 
$(t,x,y)\in (0,\infty)\times U\times U$ and we have
\begin{equation}\label{Doobph}
p^D_U(t,x,y)=  h(x)h(y)p^{D,U}_h(t,x,y).
\end{equation}
Clearly, to obtain good estimates for $p^D_U$, it suffices 
to obtain good estimates for $p^{D,U}_h$.

\subsection{Local Dirichlet heat kernel estimates}

In this subsection, we explain how to implement the strategy outlined above
to obtain heat kernel estimates for $p^D_U(t,x,y)$ at
two fixed points $x$ and  $y$ in $U$ (these points may, in some sense, 
be close to the boundary).  

We fix $R_x,R_y>0$ and make the following two basic assumptions:
\begin{enumerate}
\item  The volume doubling property and the Poincar\'e inequality hold
up to scale $R_x$ in $B(x,R)$ and up to scale $R_y$ in $B(y,R_y)$.

\item The domain $U$ is locally $(c_u,C_u)$-inner uniform up to scale $R_x$
near $B_U(x,R)$ and up to scale $R_y$ near $B_U(y,R_y)$.

\end{enumerate}

Next, we pick a real $\gamma$ with the property that there exists a function
$h=h_\gamma$ such that $h$ is positive continuous in $U$ and is
 a $(U,B_U(x,R))$-profile and a $(U,B_U(y,R))$-profile for $\e+\gamma\langle\cdot,\cdot\rangle$.

The following lemma provides the existence of such a pair $(\gamma,h_\gamma)$.
\begin{lemma} Let $\e$ be a form satisfying {\em Assumption A}. 
Let $U$ be a domain in $X$. Set
$$\lambda_U=\inf\{ \e(f,f), f\in \F^0(U), \|f\|_2=1\}.$$ 
\begin{enumerate}
\item If $U$ is bounded then 
$-\lambda_U$ is an eigenvalue  for the infinitesimal generator of
$P^D_{U,t}$ and the associated  normalized $L^2$-eigenfunction $\phi=\phi_U\in \F^0(U)$ 
is positive in $U$.
\item If $U$ is unbounded and locally inner uniform, 
then there exists a function $h=h_U$ which is positive 
continuous in $U$, a local weak solution of $-Lh=\lambda_U h$ in $U$, and 
both a $(U,B_U(x,R_x))$-profile and a $(U,B_U(y,R_y))$-profile for $\e-\lambda_U\langle\cdot,\cdot\rangle$.
\end{enumerate}
\end{lemma}
\begin{proof} Part one follows easily from Jentzsch's Theorem 
(see, e.g., \cite[Theorem V.6.6]{Schaefer}. 

For part two, we consider a relatively compact increasing exhaustion $U_n$ of 
$U$ such that $B_U(x,R_x)$ and $B_U(y,R_y)$ are contained in $U_1$.
For each $U_n$, we have an eigenfunction $\phi_{U_n}$ with eigenvalue 
$\lambda_{U_n}$ in $U_n$ given by (i). Fix a point $o\in U_1$ and consider 
the sequence $h_n=\phi_{U_n}/\phi_{U_n}(o).$ From the definitions and
\cite{LierlSC2}, it easily follows 
that these functions all satisfy local Harnack inequalities 
(with constants independent of $n$) in their domains and are 
equicontinuous. This implies that some subsequence of $(h_n)$ converges in $U$ 
to a function $h\in \F_{\mbox{\tiny loc}}(U)$ which is positive and a local 
weak solution of $-L h=\lambda_U h$
in $U$. In addition, by Theorem \ref{th-BHP2}, the functions 
$h_n$ satisfy the geometric
boundary Harnack principle locally in $B_U(x,R_x)$ and $B_U(y,R_y)$, 
uniformly in $n$. This easily implies that  the limit $h$ is a 
$(U,B_U(x,R_x))$-profile and a $(U,B_U(y,R_y))$-profile for $\e-\lambda_U\langle\cdot,\cdot\rangle$.  See 
\cite[Section 4.3.2]{GyryaSC}.
\end{proof}
\begin{remark} 
\begin{enumerate}\item 
Recall that Assumption A implies that there exists a non-negative
real $\alpha$ such that $\e(f,f)\ge -\alpha\|f\|_2^2$ for all  $f\in \F$,  and that 
$\alpha$ is bounded above in terms of $C_2(\e)$ and $C_3(\e)$. Hence, 
$\lambda_U \geq - \alpha$. 
\item It is not hard 
to modify the proof of (ii) to 
show that for each $\gamma\le \lambda_U$ there exists
a function $h_\gamma$ which is positive 
continuous in $U$, a local weak solution of $-Lh=\gamma h$ in $U$, and 
both a $(U,B_U(x,R_x))$-profile 
and a $(U,B_U(y,R_y))$-profile for $\e-\gamma\langle\cdot,\cdot\rangle$.
\end{enumerate}
\end{remark}

\begin{theorem}\label{th-HK1}
Let $\e$ be a form satisfying \emph{Assumption A}.
Let $U$ be a domain in $X$. Fix $x,y\in U$, $T>0$ and $\gamma$ and assume that
\begin{enumerate}
\item  The volume doubling property and the Poincar\'e inequality hold
up to scale $R_x$ in $B(x,R_x)$ and up to scale $R_y$ in $B(y,R_y)$.
\item $U$ is locally $(c_u,C_u)$-inner uniform up to scale $R_x$
near $B_U(x,R_x)$ and up to scale $R_y$ near $B_U(y,R_y)$.
\item There exists a function
$h=h_\gamma$ such that $h$ is positive continuous in $U$ and
 both a $(U,B_U(x,R_x))$-profile and a $(U,B_U(y,R_y))$-profile for $\e+\gamma\langle\cdot,\cdot\rangle$.
\end{enumerate}
Then for all $t\in (0,T)$, $\xi\in B_{\widetilde{U}}(x,a_0R_x)$, $\zeta\in 
B_{\widetilde{U}}(y,a_0R_y)$, we have 
$$p^D_U(t,\xi,\zeta)\le \frac{A_1h(\xi)h(\zeta)\exp(-a_1 d_U(\xi,\zeta)^2/t)}
{\sqrt{V(\xi,r_x)V(\zeta,r_y)}h(\xi_{r_x})h(\zeta_{r_y})},$$
where $r_z=\min\{\sqrt{t},a_0R_z\}$ for  $z=x,y$, and where
$z_r\in U$ denotes a point such that $d_U(z,z_r)=r/4$ and 
$d(z_r,\partial U)\ge c_u r/8$, for $z = \xi,\zeta$ and $r=r_x, r_y$. Further,
for all $t\in (0,T)$, $\xi \in B_{\widetilde{U}}(x,a_0R_x)$, $\zeta\in B_{\widetilde{U}}(y,a_0R_y)$, we have 
$$p^D_U(t,\xi,\zeta)\ge \frac{a_2 h(\xi)h(\zeta)\exp(-A_2 d_U(\xi,\zeta)^2/t)}
{\sqrt{V(\xi,r_x)V(\zeta,r_y)}h(\xi_{r_x})h(\zeta_{r_y})}.$$
The constant $a_0$ depends only on $c_u,C_u$.
The constant $a_1$ depends only on
$C_0(\e)$--$C_5(\e)$. The constants $A_1,A_2,a_2$ depend only on $c_u,C_u$, 
the volume doubling and Poincar\'e constants up to scale $R_x$ (resp. $R_y$)
in  $B(x,R_x)$ (resp. $B(y,R_y)$), $C_0(\e)$-$C_5(\e)$
and upper bounds on $|\gamma|$  and $(C_8(\e)+|\gamma|)R_x^2$,  
$(C_8(\e)+|\gamma|)R_y^2$, $TR_x^{-2}$ and $TR_y^{-2}$.  
\end{theorem}
\begin{proof} By (\ref{Doobph}) these bounds can be deduced from
similar heat kernel bounds for $p^{D,U}_{h}$. The desired bounds for $p^{D,U}_h$ 
follow from classical arguments 
(e.g., \cite[Chapter 5]{SC02} and \cite{SturmII,SturmIII})  
based on the validity of 
the parabolic Harnack inequality in 
$B_{\widetilde{U}}(x,a_0R_x)$ and $B_{\widetilde{U}}(y,a_0R_y)$ which follows 
from Theorem \ref{h-VDP}. 
\end{proof}
\begin{remark} Theorem \ref{th-HK1} holds true if we replace (iii) 
by the assumption that $h$  is positive continuous in $U$ and both
a $(U,B_U(x,R_x))$-profile and a $(U,B_U(y,R_y))$-profile for 
$\e^{\mbox{\tiny s}}+\gamma\langle\cdot,\cdot\rangle$. 
\end{remark}

\subsection{Dirichlet heat kernel estimates in unbounded domains}

In this section we prove  two-sided 
Dirichlet heat kernel estimates in an unbounded 
domain $U$ under various hypotheses. The technique of the proof is 
the same as in the previous section. 
The form $\e$ is a form as in Assumption A. Our minimal assumption 
on the unbounded domain $U$ is that the volume doubling property and 
Poincar\'e inequality hold locally on $\overline{U}$ and that $U$ is locally 
inner uniform.  By \cite[4.3.2]{GyryaSC},
these minimal hypotheses imply the existence of a 
$\e^{\mbox{\tiny s}}$-harmonic profile $h$ in $U$. Since $\e^{\mbox{\tiny s}}$
is our model form, it is natural to think of the 
$\e^{\mbox{\tiny s}}$-profile $h$ as a fundamental object. 
Hence, the heat kernel 
estimates given in this section are stated in terms of $h$.

The first theorem of this section provides heat kernel bounds 
under these minimal hypotheses.
In the second theorem, these minimal hypotheses are upgraded to  
hypotheses that hold uniformly up to scale $R$ for some fixed $R>0$. 

\begin{theorem}\label{th-HK2}
Let $\e$ be a form satisfying \emph{Assumption A}.
Let $U$ be a locally inner uniform unbounded domain in $X$. Assume that
the volume doubling property and the Poincar\'e inequality hold
locally on $\overline{U}$. Let 
$h$ be a $\e^{\mbox{\em\tiny s}}$-harmonic profile in $U$ 
(extended as a $\F$-quasi-continuous function on $\widetilde{U}$).
Then the Dirichlet heat kernel $p^D_U$ has the following properties:
\begin{enumerate}
\item
The function $(t,x,y)\mapsto p^D_U(t,x,y)$ is continuous  
on $(0,\infty)\times U\times U$ and the function 
$$(t,x,y)\mapsto \frac{p^D_U(t,x,y)}{h(x)h(y)}$$ is locally H\"older continuous
in 
$(0,\infty)\times\widetilde{U}\times \widetilde{U}$.
\item There exist $ c,C>0$ such that for any pair of points 
$x,y\in \widetilde{U}$, 
there exist  $r_x,r_y, A=A(x,y)\in (0,\infty)$ 
such that for any $t>0$ and $\xi\in B_{\widetilde{U}}(x,r_x)$, 
$\zeta\in B_{\widetilde{U}}(y,r_y)$, we have
$$p^D_U(t,\xi,\zeta)\le \frac{A h(\xi)h(\zeta) \exp(- cd_U(\xi,\zeta)^2/t
+C t)}{\sqrt{V(\xi,\sqrt{t_x})V(\zeta,\sqrt{t_y})} h(\xi_{\sqrt{t_x}})h(\zeta_{\sqrt{t_y}})},$$
where $t_z=\min\{ t,r_z^2\}$ for $z = x,y$, and where $z_r$ is a point in $U$ at 
distance at most $r/4$ from $z$ and at distance at least 
$c_u r/8$ from $\partial U$ for $z = \xi,\zeta$.    
\end{enumerate}
\end{theorem}

\begin{theorem}\label{th-HK3}
Let $\e$ be a form satisfying \emph{Assumption A}.
Let $U$ be an unbounded domain in $X$ that is locally 
$(c_u,C_u)$-inner uniform up to scale 
$R$. Assume that
the volume doubling property and the Poincar\'e inequality hold
up to scale $R$  on $\overline{U}$. Let 
$h$ be a $\e^{\mbox{\em\tiny s}}$-harmonic profile in $U$ 
(extended as a $\F$-quasi-continuous function on $\widetilde{U}$).
Then the Dirichlet heat kernel $p^D_U$ has the following properties:
\begin{enumerate}
\item
The function $(t,x,y)\mapsto p^D_U(t,x,y)$ is continuous  
on $(0,\infty)\times U\times U$ and there exist $\kappa, a_0\in (0,1)$ 
and a constant $A_1$ such that for any $r\in (0,a_0R)$, $t,t'\in (0, \infty)$, 
$x,x',y,y'\in \widetilde{U}$ satisfying $t\ge r^2$, $|t-t'|\le r^2/4$, 
$d_U(x,x')\le r$, $d(y,y')\le r$, we have 
$$\left|\frac{p^D_U(t,x,y)}{h(x)h(y)}-
\frac{p^D_U(t',x',y')}{h(x')h(y')}\right|\le A_1  \left(\frac{\rho}{r}
\right)^\kappa \frac{p^D_U(t+r^2,x_r,y_r)}{h(x_r)h(y_r)}.
$$ 
where $\rho=\sqrt{|t-t'|}+
d_U(x,x')+d_U(y,y')$ and $z_r$ is a point
in $U$ at 
distance at most $r/4$ from $z$ and at distance at least 
$c_u r/8$ from $\partial U$ for $z = x,y$.    

\item There exist $ c,C, a_0,A_1, a_2A_2>0$ such that for any pair of points 
$x,y\in \widetilde{U}$ and any $t>0$, we have
$$p^D_U(t,x,y)\le \frac{A_1 h(x)h(y) \exp(- cd_U(x,y)^2/t
+C t)}{\sqrt{V(x,\sqrt{\tau})V(y,\sqrt{\tau})} h(x_{\sqrt{\tau}})
h(y_{\sqrt{\tau}})}
$$ 
and
$$p^D_U(t,x,y)\ge \frac{a_2 h(x)h(y) \exp(- A_2 d_U(x,y)^2/t
-A_2 t)}{\sqrt{V(x,\sqrt{\tau})V(y,\sqrt{\tau})} h(x_{\sqrt{\tau}})
h(y_{\sqrt{\tau}})}
$$ 
where $\tau=\min\{ t, (a_0R)^2\}$.    
\end{enumerate}
The constant $a_0$ depends only on $c_u,C_u$.
The constants $c,C$ depend only on $C_2(\e)$--$C_5(\e)$.
The constants $\kappa, A_1,a_2,A_2$ depend only on $c_u,C_u$, 
the volume doubling and Poincar\'e constant on $\overline{U}$ up to scale $R$,
$C_0(\e)$--$C_5(\e)$,
and on an upper bound on $C_8(\e)R^2$.
\end{theorem}
\begin{proof}[Proof of Theorems \ref{th-HK2}--\ref{th-HK3} (outline)] The proofs 
of the two theorems stated above follow well established lines of reasoning.
The first (and crucial) step is to use (\ref{Doobph}) and estimate the kernel 
$p^{D,U}_h$. Indeed, by Theorem \ref{h-VDP}, the associated form $\e_h$
satisfies a parabolic Harnack inequality. The desired bounds for $p^{D,U}_h$ 
follow from classical arguments 
(e.g., \cite[Chapter 5]{SC02} and \cite{SturmII,SturmIII})  
based on the validity of the parabolic Harnack inequality.
\end{proof}

\begin{remark}In statement {\em (ii)} of Theorem \ref{th-HK3}, 
the  denominators can be replaced by
$$V(x,\sqrt{\tau})[h(x_{\sqrt{\tau}})]^2.$$ Note the lack of $x,y$ 
symmetry of the resulting bounds. This is often useful in practice.
\end{remark}
The following corollary of the Harnack inequality for $\e_h$ 
up to scale $a_0R$ in $\widetilde{U}$ is also of interest. We note that if 
$u\in \F^0_{\mbox{\tiny loc}}(U,(0,\infty)\times U)$
is a local weak solution of the heat equation for $\e$ in $U$, then
$u/h \in \F^{h,0}_{\mbox{\tiny loc}}(U,(0,\infty)\times U)$ 
is a local weak solution of the heat equation for $\e_h$ in $U$.
Hence $u/h$ satisfies the Harnack inequality up to scale $a_0R$ in 
$\widetilde{U}$. This and the argument given in \cite[Section 5.4.3]{SC02}
yield the following result.

\begin{theorem} \label{th-globH}
Let $\e$ be a form satisfying \emph{Assumption A}.
Let $U$ be an unbounded domain in $X$ that is locally 
$(c_u,C_u)$-inner uniform up to scale 
$R$. Assume that
the volume doubling property and the Poincar\'e inequality hold
up to scale $R$  on $\overline{U}$. Let 
$h$ be a $\e^{\mbox{\em\tiny s}}$-harmonic profile in $U$ 
(extended as a $\F$-quasi-continuous function on $\widetilde{U}$).
Let $u$ be a positive local weak solution of the heat equation 
for $\e$ in $U$
with Dirichlet boundary condition along $\partial U$.
Then there exists a constant $A_1$ such that for all $0<s<t<\infty$ and 
$x,y\in \widetilde{U}$, we have
$$\frac{u(s,x)}{u(t,y)}\le A_1 \frac{h(x)}{h(y)}\exp\left( A_1\left( 1+ \frac{t-s}{s}+\frac{t-s}{R^2}+ \frac{d_U(x,y)^2}{t-s}\right)\right).$$
The constant $A_1$ depends only on $c_u,C_u$, 
the volume doubling and Poincar\'e constant on $\overline{U}$ up to scale $R$,
$C_0(\e)$--$C_5(\e)$,
and on an upper bound on $C_8(\e)R^2$.
\end{theorem}

\subsection{Dirichlet heat kernel estimates in bounded domains}\label{sec-conv}
This section focuses on estimates in bounded inner uniform domains and 
relates these results to refined intrinsic ultracontractivity estimates.

Very generally, consider a positivity preserving strongly continuous 
semigroup $P_t$ acting on $L^2(U,\mu)$, where $U$ is a bounded domain, 
with continuous kernel $p(t,x,y)$ 
such that $p(t,x,y)$ is bounded for each $t>0$. 
Its adjoint $P^*_t$ (with kernel $p^*(t,x,y)=p(t,y,x)$) 
has the same properties. Let $\lambda_U$ be the common bottom of the 
$L^2$-spectrum of $-L$ and $-L^*$ where $L$ and $L^*$ are the respective 
infinitesimal generators. Let $\phi$ and $\phi_*$ be the 
associated positive continuous $L^2$-normalized
eigenfunctions. Following \cite{KimSong},
we say that the pair $(P_t,P^*_t)$ is intrinsically ultracontractive
if for each $t>0$ there exists a constant $c(t)$ such that
\begin{equation} \label{ultract}
p(t,x,y)\le c(t) \phi(x)\phi_*(y).
\end{equation}
For selfadjoint semigroups, intrinsic ultracontractivity was 
introduced in \cite{DavSim}. Note that if $\lambda_\psi$ 
is an eigenvalue for $P_t$ 
with $L^2$-normalized eigenfunction $\psi$ then (\ref{ultract}) implies
\begin{equation} \label{eig}
 |\psi|\le e c(1/|\lambda_\psi|)^{1/2}
\phi  \end{equation}
In many interesting cases, these bounds hold with 
$c_t= c(1+ t^{-\nu/2})e^{-t\lambda_U}$ for some $\nu>0$. Typically, 
in the literature, $U$ is a domain in $\mathbb R^n$ and 
$P_t$ is the semigroup associated with an elliptic  
second order differential operator (e.g., the Laplacian) 
with Dirichlet boundary condition along 
the boundary of $U$. Intrinsic ultracontractivity is then viewed as 
a property that depends on the regularity of the boundary of $U$. 
See, e.g.,  \cite{Ban,BanDav}. In particular, it follows from \cite{Ban}
that the heat semigroup with Dirichlet boundary condition 
in any bounded  inner uniform domain $U\subset \mathbb R^n$ is 
intrinsically ultracontractive with $c_t=c(1+t^{-\nu/2})e^{-\lambda_U t}$ 
for some $c=c(U), \nu=\nu(U)$. Here, we obtain the following refined results.

\begin{theorem}\label{th-bounded1}
Let $\e$ be a form satisfying \emph{Assumption A}.
Let $U$ be a bounded domain in $X$ that is  locally
$(c_u,C_u)$-inner uniform up to scale $R$. Assume that
the volume doubling property and the Poincar\'e inequality hold
up to scale $R$ on $\overline{U}$. Let 
$$\lambda=\lambda_U=\min\{\e(f,f): f\in \F^0(U), \|f\|_2=1\},$$
and  let $\phi=\phi_U$ be the associated positive $L^2$-normalized eigenfunction
(of minus the infinitesimal generator 
with Dirichlet boundary condition along $\partial U$).
Then, for all $ t\in (0,R^2)$, $x,y\in \widetilde{U}$, 
the Dirichlet heat kernel $p^D_U$ satisfies
\begin{equation}\label{phi1}
p^D_U(t,x,y)\le  \frac{A_1
\phi(x)\phi(y)e^{  -cd_U(x,y)^2/t}}{\sqrt{V(x,\sqrt{t})V(y,\sqrt{t})}
\phi(x_{\sqrt{t}})\phi(y_{\sqrt{t}})}
\end{equation}
and
\begin{equation}\label{phi2}
p^D_U(t,x,y)\ge  \frac{a_2
\phi(x)\phi(y)e^{ -A_2d_U(x,y)^2/t}}{\sqrt{V(x,\sqrt{t})V(y,\sqrt{t})}
\phi(x_{\sqrt{t}})\phi(y_{\sqrt{t}})}
\end{equation}
Further, for $t>R^2$, we have
\begin{equation} \label{phi3}
a_3
\le
\frac{e^{\lambda t}
p^D_U(t,x,y)}{\phi(x)\phi(y)} \le A_3 
\end{equation}
The constant $c$ depends only on $C_0(\e)$--$C_5(\e)$.
The constants $A_1,a_2,A_2, a_3,A_3\in (0,\infty)$ 
depend only on $c_u,C_u$, the volume doubling and Poincar\'e constants on 
$\overline{U}$ up to scale $R$,
$C_0(\e)$--$C_5(\e)$, and on  upper bounds on $(C_8(\e)+|\lambda|)R^2$ and 
$\mathrm{diam}_U/R$.
\end{theorem}

\begin{corollary}\label{cor-bounded1}
Referring to the notation and setting of 
{\em Theorem \ref{th-bounded1}}, there exist a bounded continuous 
function $w$ on $U$ and a real $\omega>0$ such that
\begin{equation}\label{cor-phi}
\forall\,t\ge R^2, \;\;x,y\in U,\;\;\;\left|\frac{e^{\lambda t}
p^D_U(t,x,y)}{\phi(x)\phi(y)w(y)} -1\right|\le A_4 e^{-\omega t}.\end{equation}
Further,  $a_3\le w\le A_3$, 
$$A_4\le\frac{A_3}{a_3(1-a_3/A_3)^{2}} \;\mbox{ and }\;\;
\omega \ge \frac{1}{R^2}\log\left(\frac{1}{1-a_3/A_3}\right)$$ where $a_3,A_3 $ and $R$ are 
as in {\em Theorem \ref{th-bounded1}}.
\end{corollary} 
\begin{proof}
By definition, the semigroup 
$K_t=e^{\lambda t}P^{D,U}_{\phi,t}$ with kernel $$
K_t(x,y)=\frac{e^{\lambda t}
p^D_U(t,x,y)}{\phi(x)\phi(y)}$$ with respect to $\phi^2d\mu$
is positivity preserving and satisfies
$K_t\mathbf 1_U=\mathbf 1_U$.
It follows that its adjoint $K^*_t$ on $L^2(U,\phi^2d\mu)$ 
admits a positive continuous eigenfunction 
$w$ with eigenvalue $1$. We normalize $w$ by setting $\int w\phi^2d\mu=1$.
Obviously, 
$w\phi^2d\mu$ is then an invariant probability measure for $K_t$ 
and it follows from
(\ref{phi3}) that $w$ is bounded and bounded away from $0$.

In the following computation, we think of  $K_t$ and $w$ as Markov operators, 
namely, 
$$f\mapsto K_tf=\int K_t(\cdot,y) f(y) \phi(y)^2d\mu(y),\;\;
f\mapsto wf=\int fw \phi^2d\mu,$$ 
acting on $L^p(U,w\phi^2d\mu)$. Note that (\ref{phi3}) implies
$a_3\le w\le A_3$. Hence
there exists a constant $\epsilon=a_3/A_3>0$ such that 
$K_{R^2}(x,y)\ge \epsilon w(y).$   It follows that 
$Q(x,y)=(1-\epsilon)^{-1}(K_{R^2}(x,y)-\epsilon w(y))$ is a Markov kernel on $U$
with respect to $\phi^2d\mu$ and we again denote by $Q$ the associated 
operator acting on $L^p(U,w\phi^2d\mu)$. Since $w\phi^2d\mu$ is an 
invariant probability measure for $Q$, we have  $ (Q-w)^n= Q^n(I-w)$.
Note also that, since $Q-w= (1-\epsilon)^{-1}(K_{R^2}-w)$,
$$\sup_{x,y}\{|Q^n(x,y)/w(y)-1|\}=\|Q^n(I-w)\|_{1\rightarrow  \infty}$$
where the right-hand side is the norm of the operator $Q^n(I-w)=Q^{n-1}(Q-w)$
from $L^1(U,w\phi^2d\mu)$ to $L^\infty(U,w\phi^2d\mu)$. We have
$$\|Q-w\|_{1\rightarrow \infty}\le \epsilon^{-1}(1-\epsilon)^{-1} \;\mbox{ and }\;\; 
\|Q^{n-1}\|_{1\rightarrow 1}\le 1.$$
Hence, we obtain 
$$\sup_{x,y}\{|Q^n(x,y)/w(y)-1|\}\le \epsilon^{-1}(1-\epsilon)^{-1}. $$
Since $Q^n(I-w) = (Q-w)^n= (1-\epsilon)^{-n}(K_{nR^2}-w)$, this gives
$$ 
\sup_{x,y}\{|K_{nR^2}(x,y)/w(y)-1|\}
\le \epsilon^{-1}(1-\epsilon)^{n-1}.$$
Since $t\mapsto \sup_{x,y}\{|K_{t}(x,y)/w(y)-1|\}$ is non-increasing in $t$,
we obtain
$$ 
\sup_{x,y}\{|K_{t}(x,y)/w(y)-1|\}
\le \epsilon^{-1}(1-\epsilon)^{-2} e^{-\omega t},
\;\;\omega= -R^{-2}\log (1-\epsilon).$$
This is exactly the desired inequality.
\end{proof}
\begin{remark} Let $\phi_*$ be the positive eigenfunction associated with 
the bottom eigenvalue $\lambda$ for the adjoint $-L^*$
of the infinitesimal generator $-L$ of $P^{D}_{U,t}$. From the definitions 
of $\phi,\phi_*,w$,
we  deduce that   
$\phi_*= w\phi$
so that we can rewrite (\ref{cor-phi}) as 
\begin{equation}\label{cor-phi*}
\forall\,t\ge R^2, \;\;x,y\in U,\;\;\;\left|\frac{e^{\lambda t}
p^D_U(t,x,y)}{\phi(x)\phi_*(y)} -1\right|\le A_4 e^{-\omega t}.\end{equation}
Further, we have $c \phi\le \phi_*\le C\phi$ for some positive constants $c,C$.
\end{remark}

\begin{corollary} \label{cor-bounded2}
Referring to the notation and setting of 
{\em Theorem \ref{th-bounded1}}, there exists a constant $A_5$ such that,
if $\psi \neq \phi$ is an $L^2(U,\mu)$-normalized 
eigenfunction of $-L$ with eigenvalue $\lambda_\psi$ then
$\eta=\lambda_\psi-\lambda\ge 1/(A_5R^2)$ and
\begin{equation}\label{phi5}
\forall\,x\in U,\;\;|\psi(x)|
\le \frac{ A_5\,\phi(x)}{\sqrt{V(x,1/\sqrt{\eta})}\phi(x_{1/\sqrt{\eta}})}.
\end{equation}
The constant $A_5$
depends only on $c_u,C_u$, the volume doubling and Poincar\'e constants on 
$\overline{U}$ up to scale $R$,
$C_0(\e)$--$C_5(\e)$, and on  upper bounds on 
$(C_8(\e)+|\lambda|)R^2$ and 
$\mathrm{diam}_U/R$.
\end{corollary}
\begin{proof}By hypothesis, we have
$P^{D}_{U,t}\psi=e^{-t\lambda_\psi}\psi$. Hence
$$ e^{\lambda t}P^{D,U}_{\phi,t}(\psi/\phi)= 
e^{(\lambda-\lambda_\psi)t}(\psi/\phi).$$ 
The previous corollary implies that 
$$\lambda_\psi-\lambda \ge \omega 
=1/(A_5R^2)$$ with $A^{-1}_5=\log(1-a_3/A_3)^{-1}$. Further, 
for any $x\in U$ and  $t\le R^2$, (\ref{phi1}) yields
$$\int |p^{D,U}_\phi(t,x,y)|^2\phi(y)^2d\mu(y)\le 
\frac{ A'_1}{V(x,\sqrt{t}) \phi(x_{\sqrt{t}})^2}$$
where $A'_1$ depends on the same constants 
as $A_1$ in Theorem \ref{th-bounded1}. Because $\int \psi^2\phi^2d\mu=1$, it  
follows that
\begin{eqnarray*}
e^{(\lambda-\lambda_\psi)t}\frac{|\psi(x)|}{\phi(x)}&=&
e^{\lambda t}\left|
\int p^{D,U}_\phi(t,x,y)\frac{\psi(y)}{\phi(y)} \phi(y)^2d\mu(y)\right|\\
&\le & 
e^{\lambda t}\left(\int |p^{D,U}_\phi(t,x,y)|^2\phi(y)^2d\mu(y)\right)^{1/2} 
\\
&\le & \frac{\sqrt{A'_1} e^{t\lambda} }{\sqrt{V(x,\sqrt{t})} 
\phi(x_{\sqrt{t}})}
\end{eqnarray*}
It now suffices to choose $t\simeq 1/(\lambda_\psi-\lambda)=1/\eta$  
(which is, indeed, of order at most $R^2$) to obtain 
$$|\psi(x)|\le \frac{\sqrt{A'_1} e^{|\lambda|R^2}\, \phi(x)}{\sqrt{V(x,1/\sqrt{\eta})}\, 
\phi(x_{1/\sqrt{\eta}})}.$$
\end{proof}

The following result provides a very useful comparison between the 
principal Dirichlet  eigenfunction $\phi$ associated to $\e$ in $U$
and the principal Dirichlet eigenfunction $\phi_{\mbox{\tiny s}}$ 
associated to $\e^{\mbox{\tiny s}}$ in $U$.   Recall that
$$\lambda=\lambda_U=\min\{\e(f,f): f\in \F^0(U), \|f\|_2=1\},$$
and set
$$\lambda_{\mbox{\tiny s}}=\lambda
_{{\mbox{\tiny s}},U}=\min\{\e^{\mbox{\tiny s}}(f,f): f\in \F^0(U), \|f\|_2=1\}.$$
Assumption A on the form $\e$ implies easily that there 
exists a constant $A$ such that
$$\frac{1}{2}\lambda_{\mbox{\tiny s}}-A \le \lambda\le 
\lambda_{\mbox{\tiny s}}+A.$$
Further, under the assumption of Theorem \ref{th-bounded1}, 
there exists a constant $A'$ such that  
$0\le \lambda_{\mbox{\tiny s}}\le A'/R^2$. Here $A'$ depends on $c_u,C_u$ and 
the doubling constant up to scale $R$ on $\overline{U}$.
\begin{theorem}
Referring to the notation and setting of 
{\em Theorem \ref{th-bounded1}}, there exists a constant $A_6$ such that
the principal Dirichlet  eigenfunction $\phi$ associated to $\e$ 
and the principal Dirichlet eigenfunction $\phi_{\mbox{\emph{\tiny{s}}}}$ 
associated to $\e^{\mbox{\emph{\tiny{s}}}}$ in $U$ satisfy
$$ A_6^{-1} \phi_{\mbox{\emph{\tiny{s}}}}\le \phi\le A_6 \phi_{\mbox{\emph{\tiny{s}}}}.$$
The constant $A_6$ depends only on $c_u,C_u$, 
the volume doubling and Poincar\'e constants on 
$\overline{U}$ up to scale $R$,
$C_0(\e)$-- $C_5(\e)$, and on  upper bounds on 
$C_8(\e)R^2$ and 
$\mathrm{diam}_U/R$.
\end{theorem} 
\begin{proof} Apply Theorem \ref{h-VDP}(i) 
with $h=\phi_{\mbox{\tiny{s}}}$
(hence $\gamma= \lambda_{\mbox{\tiny s}}$ and $|\gamma|R^2$ is bounded above by 
the constant $A'$ appearing just before the theorem). Now, 
$\phi/\phi_{\mbox{\tiny s}}$ is a harmonic function for the form 
$\e_{\phi_{\mbox{\tiny s }}} -\lambda\langle\cdot,\cdot\rangle$ and the 
corresponding Harnack inequality provided by Theorem \ref{h-VDP}(i) gives 
the desired result.
\end{proof}

\begin{remark} In Theorem \ref{th-bounded1}, Corollary \ref{cor-bounded1} 
and Corollary \ref{cor-bounded2}, consider the special case when the 
volume doubling property and Poincar\'e inequality hold globally on 
$(X,(\e^{\mbox{\tiny s}},\F),d,\mu)$. Specialize further to the case when 
$\e=\e^{\mbox{\tiny s}}$. Assume that $U$ is a 
$(c_u,C_u)$-inner uniform domain in $(X,d)$. Then 
(\ref{phi1})-(\ref{phi2})-(\ref{phi3}) and (\ref{cor-phi})-(\ref{phi5})
hold true with $R=\mathrm{diam}_U$ and 
constants $A_1,a_2,A_2,a_3,A_3,A_4,A_5$ depending only on $c_u,C_u$.
\end{remark}

\def\cprime{$'$} \def\cprime{$'$}

\end{document}